%% file: AccIHT.tex
\documentclass[11pt]{amsart}
\usepackage{geometry}                % See geometry.pdf to learn the layout options. There are lots.
\input{preamble}

\usepackage{epstopdf}

\title[Provable convergence for accelerated IHT]{IHT \MakeLowercase{dies hard:} P\MakeLowercase{rovable convergence for accelerated} I\MakeLowercase{terative} H\MakeLowercase{ard} T\MakeLowercase{hresholding}}
\author[Khanna, Kyrillidis]{R\MakeLowercase{ajiv} K\MakeLowercase{hanna}$^\dagger$ \MakeLowercase{and} A\MakeLowercase{nastasios} K\MakeLowercase{yrillidis}$^\star$   \\ $^\dagger$U\MakeLowercase{niversity of }T\MakeLowercase{exas at }A\MakeLowercase{ustin} \\ $^\star$IBM T.J. W\MakeLowercase{atson} R\MakeLowercase{esearch} C\MakeLowercase{enter}}

\begin{document}

\maketitle

\begin{abstract}
We study --both in theory and practice-- the use of momentum in classic iterative hard thresholding (IHT) methods. 
By simply modifying classical IHT, we investigate its convergence behavior on convex optimization criteria with non-convex constraints, under standard assumptions.
We observe that acceleration in IHT leads to significant improvements, compared to state of the art projected gradient descent and Frank-Wolfe variants.
As a byproduct of our inspection, we study the impact of selecting the momentum parameter: similar to convex settings, two modes of behavior are observed --``rippling'' and linear-- depending on the level of momentum.
\end{abstract}

\begin{center}
\textit{(Improved version based on comments from peer-review process; removed assumptions on optimal point; noisy case included)}
\end{center}

\input{intro}

\input{problem}

\input{algo}

\input{theory}

\input{related}

\input{experiments}

\input{conclusions}

%\subsubsection*{Acknowledgments}

\newpage
\bibliographystyle{unsrt}
\bibliography{biblio}

\newpage
\onecolumn
\appendix
\input{proofs2}
\input{practice}
\input{conv_appendix}
\input{counter}
\input{moreexperiments}

\end{document}

%% file: preamble.tex
\usepackage[utf8]{inputenc} % allow utf-8 input
\usepackage[T1]{fontenc}    % use 8-bit T1 fonts
\usepackage{nicefrac}       % compact symbols for 1/2, etc.
\usepackage{amsmath}
\usepackage{amsfonts}
\usepackage{amssymb}

\usepackage{enumitem}
\usepackage{multirow}
\usepackage{algorithm}
\usepackage{algorithmic}
\usepackage{amsthm}
\usepackage{bm}
\usepackage{enumerate}
\usepackage{bbm}
\usepackage[table]{xcolor}
\usepackage{tikz} % for plotting graphs
\usetikzlibrary{arrows,automata}
\usepackage{setspace} % increase interline spacing slightly
\usepackage{framed}
\usepackage{verbatim}
\usepackage{wrapfig}
\usepackage{caption}
\usepackage{subcaption}
\usepackage{stackrel}
\usepackage{xfrac}


\usepackage{graphicx}
\graphicspath{{images/}}

\usepackage{booktabs}
\usepackage{lipsum}
\usepackage{microtype}
\usepackage{url}
\usepackage[final]{pdfpages}

\DeclareMathOperator*{\argmin}{arg\,min}

\def \dim {n}

\newcommand{\R}{\mathbb{R}}

\newtheorem{definition}{Definition}
\newtheorem{lemma}{Lemma}
\newtheorem{theorem}{Theorem}

\newcommand{\vectornorm}[1]{\|#1\|}

%% file: intro.tex
%!TEX root = AccIHT.tex

\section{Introduction}

It is a well-known fact in convex optimization that momentum techniques provably result into significant gains w.r.t. convergence rate.
Since 1983, when Nesterov proposed his \emph{optimal gradient methods} \cite{nesterov1983method}, these techniques have been used in diverse machine learning and signal processing tasks.
%While the use of momentum is not fully understood even in convex settings \cite{goh2017why}, acceleration techniques have lately re-gained popularity in non-convex settings, thanks to their improved performance in structured practical scenaria: from empirical risk minimization (ERM) to training neural networks. 
Lately, the use of momentum has re-gained popularity in non-convex settings, thanks to their improved performance in structured practical scenaria: from empirical risk minimization (ERM) to training neural networks.

Here, we mainly focus on structured constrained ERM optimization problems: 
\begin{equation}
	\begin{aligned}
	& \underset{x\in\R^\dim}{\text{minimize}}
	& & f(x) \left(:= \tfrac{1}{2} \|b - \Phi x\|_2^2 \right)
	& \text{subject to}
	& & x \in \mathcal{C},
	\end{aligned} \label{eq:generic_problem} 
\end{equation} 
that involve convex objectives $f$ and simple structured, but non-convex, constraints $\mathcal{C}$, that can be described using a set of atoms, as in \cite{negahban2009unified, chandrasekaran2012convex}; see also Section \ref{subsec:model}.
We note that our theory goes beyond the vector least squares case, with applications also in low-rank recovery.

Practical algorithms for \eqref{eq:generic_problem} are convexified projected gradient descent schemes \cite{chandrasekaran2012convex}, non-convex iterative hard thresholding (IHT) variants \cite{bahmani2013greedy} and Frank-Wolfe (FW) methods \cite{clarkson2010coresets}. 
Convex methods can accommodate acceleration due to \cite{nesterov1983method, nesterov2013gradient} and come with rigorous theoretical guarantees; but, higher computational complexity might be observed in practice (depending on the nature of $\mathcal{C}$);
further, their configuration could be harder and/or non-intuitive.
FW variants \cite{jaggi2013revisiting, lacoste2015global} simplify the handling of constraints, but the successive construction of estimates --by adding singleton atoms to the putative solution-- could slow down convergence. 
Non-convex methods, such as IHT \cite{blumensath2009iterative, jain2010guaranteed}, could be the methods of choice in practice, but only few schemes justify their behavior in theory. 
Even more importantly, IHT schemes that utilize acceleration inferably are lacking.
We defer the discussion on related work to Section \ref{sec:related}.

In this work, we study the use of acceleration in IHT settings and supply additional information about open questions regarding the convergence and practicality of such methods on real problems.
The current paper provides evidence that ``IHT dies hard'':
\begin{itemize}[leftmargin=0.5cm]
\item Accelerated IHT comes with 
%strong 
theoretical guarantees for the minimization problem \eqref{eq:generic_problem}. 
While recent results \cite{barber2017gradient} focus on plain IHT, there are no results on Accelerated IHT, apart from \cite{kyrillidis2014matrix}. % on specific cases of \eqref{eq:generic_problem} and under stricter assumptions.
The main assumptions made here are the existence of an exact projection operation over the structure set $\mathcal{C}$, as well as standard regularity conditions on the objective function. 
\item Regarding the effect of the momentum on the convergence behavior, our study justifies that similar --to convex settings-- behavior is observed in practice for accelerated IHT: two modes of convergence exist (``rippling'' and linear), depending on the level of momentum used per iteration.
\item We include extensive experimental results with real datasets and highlight the pros and cons of using IHT variants over state of the art for structured ERM problems.
%\item One prominent way to tackle such structured problems is through Frank-Wolfe. Describe the pros and cons of Frank-Wolfe; highlight the cons at the end to connect with IHT.
%\item Explain IHT: it is a method used in signal processing (citations) and very recently appeared also in machine learning applications. Dates back to 70's. Mention the initial approaches for IHT (constant step size, specific problems, etc) and what remained as open question. 
%\item What this paper is about: accelerated IHT, an open question regarding convergence guarantees and how it compares to state-of-the-art low complexity Frank-Wolfe methods. 
\end{itemize}

%Taken separately, none of the components in this approach is new. 
%However, their combination and application to solve structured machine learning problems leads to effective algorithms that have not previously considered.
%For instance, in the experimental section, numerical results show that, under common and reasonable assumptions, problems such as sparse linear regression, \textcolor{magenta}{add more experiments}, can be efficiently solved using our accelerated non-convex variant, which is at least competitive with state of the art and comes with strong convergence guarantees. 
%
%Furthermore, this paper provides theoretical convergence guarantees for the simple projected gradient descent algorithm, over non-convex sets, where $f$ is a general convex objective function, and a momentum term is used per iteration. 
%To the best of our knowledge, this is the first work that develops theory under these settings and shades some light w.r.t. efficiency between state of the art methods. 

%Aside from the theoretical and empirical evidence, o
Our framework applies in numerous structured applications, and one of its primary merits is its flexibility.

%The existence of simple hard thresholding operators, beyong sparse and low-rank models, allows us to easily extend the current work to more complicated low-dimensional structures. 
%In this work, we demonstrate empirically that the methods we propose can be generalized to ... and ...

%The use of hard-thresholding operations in practice (either over atoms or over collection of atoms) is widely known.
%As we explain in Section \ref{}, there is vast related work on iterative hard thresholding variants, Frank-Wolfe variants, and accelerated methods for convex and non-convex formulations.

%To be sure, we tried to be extensive in 
%
%The majority of works rely on the fact that the objective criterion is convex everywhere.
%This implies that the low-dimensional structure is induced in the solution through convex relaxations; however, the level of structure (\emph{e.g.}, the level of sparsity in sparse settings) is not known to optimal in general and depends a lot on tuning non-intuitive regularization parameters \cite{kakade2010learning}.

%% file: problem.tex
%!TEX root = AccIHT.tex

\section{Problem statement}

\subsection{Low-dimensional structures}{\label{subsec:model}}

Following the notation in \cite{chandrasekaran2012convex}, let $\mathcal{A}$ denote a set of atoms; \emph{i.e.}, simple building units of general ``signals''.
\emph{E.g.}, we write $x \in \R^n$  as $x = \sum_{i} w_i a_i$, where $w_i$ are weights and $a_i \in \R^\dim$ atoms from $\mathcal{A}$.

Given $\mathcal{A}$, let the ``norm'' function $\|x\|_{0, \mathcal{A}}$ \emph{return the minimum number of superposed atoms that result into $x$}. 
Note that $\|\cdot\|_{0, \mathcal{A}}$ is a non-convex entity for the most interesting $\mathcal{A}$ cases.
Also, define the support function $\texttt{supp}_{\mathcal{A}}(x)$ as the function that returns the indices of active atoms in $x$.
Associated with $\|\cdot\|_{0, \mathcal{A}}$ is the projection operation over the set $\mathcal{A}$:
\begin{align*}
\Pi_{k, \mathcal{A}}(x) \in \argmin_{y: \|y\|_{0, \mathcal{A}} \leq k} \tfrac{1}{2} \|x - y\|_2^2.
\end{align*}
%\emph{I.e.}, it finds the closest --in the Euclidean sense-- ``signal'' that has a succinct representation, according to model $\mathcal{A}$.
To motivate our discussion, we summarize some well-known sets $\mathcal{A}$ used in machine learning problems; for a more complete description see \cite{bach2012structured}.

\emph{$\mathcal{A}$ represents plain sparsity:} 
Let $\mathcal{A} = \{ a_i \in \R^\dim ~|~ a_i \equiv \pm e_i, \forall i \in [\dim] \}$, where $e_i$ denotes the canonical basis vector. 
In this case, $k$-sparse ``signals'' $x \in \R^\dim$ can be represented as a linear combination of $k$ atoms in $\mathcal{A}$: $x = \sum_{i \in \mathcal{I}} w_i a_i$, for $|\mathcal{I}| \leq k$ and $w_i \in \R_{+}$.
The ``norm'' function is the standard $\ell_0$-``norm'' and % that counts the number of nonzero entries in $x$.
$\Pi_{k, \mathcal{A}}(x)$ finds the $k$-largest in magnitude entries of $x$.

\emph{$\mathcal{A}$ represents block sparsity \cite{jain2016structured}:} 
Let $\{G_1, G_2, \dots, G_M\}$ be a collection of $M$ non-overlapping group indices such that $\cup_{i = 1}^M G_i = [\dim]$. 
With a slight abuse of notation, $\mathcal{A} = \left\{ a_i \in \R^\dim ~|~ a_i \equiv \cup_{j: j \in G_i}e_j \right\}$ is the collection of grouped indices, according to $\{G_1, G_2, \dots, G_M\}$.
Then, $k$-sparse block ``signals'' $x \in \R^\dim$ can be expressed as a weighted linear combination of $k$ group atoms in $\mathcal{A}$. %: $x = \sum_{i \in \mathcal{I}} w_i \odot a_i$, for $|\mathcal{I}| \leq k$, and $w_i \in \R^\dim$ with nonzero entries indexed by $G_i$.
The ``norm'' function is the extension of $\ell_0$-``norm'' over group structures, and %, and counts the number of active groups in $x$.
%The projection operator 
$\Pi_{k, \mathcal{A}}(x)$ finds the $k$ most significant groups (\emph{i.e.}, groups with largest energy).

\emph{$\mathcal{A}$ denotes low rankness:} Let $\mathcal{A} = \{ a_i \in \R^{m \times n}~|~ a_i = u_i v_i^\top, \|u_i\|_2 = \|v_i\|_2 = 1 \}$ be the set of rank-one matrices. 
Here, sparsity corresponds to low-rankness.
%: $k$-rank ``signals'' $x \in \R^{m \times n}$ can be represented as a linear combination of $k$ rank-1 matrices in $\mathcal{A}$: $x = \sum_{i \in \mathcal{I}} w_i a_i$, for $|\mathcal{I}| \leq k$ and $w_i \in \R_{+}$.
The ``norm'' function corresponds to the notion of rankness;
%The projection operator 
$\Pi_{k, \mathcal{A}}(x)$ finds the best $k$-rank approximation.

\subsection{Loss function $f$}

Let $f: \R^\dim \rightarrow \R$ be a differentiable convex loss function. %, satisfying the following two standard assumptions in ML.
We consider applications that can be described by \emph{restricted strongly convex and smooth} functions $f$. % with \emph{restricted gradient Lipschitz continuity}. 
%We state these standard definitions below.

\begin{definition}{\label{def_00}}
Let $f$ be convex and differentiable. $f$ is $\alpha$-restricted strongly convex over $\mathcal{C} \subseteq \R^n$  if: 
\begin{small}
\begin{equation}\label{eq:sc}
f(y) \geq f(x) + \left \langle \nabla f(x),~ y - x \right \rangle + \tfrac{\alpha}{2} \|x - y\|_2^2, ~\forall x, y \in  \mathcal{C}.
\end{equation}
\end{small}
\end{definition}

\begin{definition}{\label{def_01}}
Let $f$ be a convex and differentiable. $f$ is $\beta$-restricted smooth over $\mathcal{C} \subseteq \R^n$ if: 
\begin{small}
\begin{equation}
f(y) \leq f(x) + \left \langle \nabla f(x),~ y - x \right \rangle + \tfrac{\beta}{2} \|x - y\|_2^2, ~\forall x, y \in  \mathcal{C}.
\end{equation} %for all $x, y \in  \mathcal{C}$.
\end{small}
\end{definition} 

Combined with the above, $\mathcal{C}$ could be the set of $rk$-sparse vectors, $rk$-sparse block ``signals'', etc, for some integer $r > 0$.
In this case, and for our theory, Definitions \ref{def_00} and \ref{def_01} simplify to the generalized RIP definition for least-squares objectives $f(x) = \tfrac{1}{2} \|b - \Phi x\|_2^2$:
\begin{align*}
\alpha_{rk} \|x\|_2^2 \leq \|\Phi x\|_2^2 \leq \beta_{rk} \|x\|_2^2
\end{align*}
where $\alpha_{rk}$ is related to the restricted strong convexity parameter, and $\beta_{rk}$ to that of restricted smoothness.

\subsection{Optimization criterion}
Given $f$ and a low-dimensional structure $\mathcal{A}$, we focus on the following optimization problem:
\begin{equation}
	\begin{aligned}
	& \underset{x\in\R^\dim}{\text{minimize}}
	& & f(x)
	& \text{subject to}
	& & \|x\|_{0, \mathcal{A}} \leq k.
	\end{aligned} \label{eq:problem}
\end{equation} 
Here, $k \in \mathbb{Z}_{+}$ denotes the level of ``succinctness''.
Examples include $(i)$ sparse and model-based sparse linear regression which is the case our theory is based on, $(ii)$ low-rank learning problems, and $(iii)$ model-based, $\ell_2$-norm regularized logistic regression tasks; see also Section \ref{sec:experiments}.

%% file: algo.tex
%!TEX root = AccIHT.tex

\section{Accelerated IHT variant}{\label{sec:algo}}

We follow the path of IHT methods. 
These are first-order gradient methods, that perform per-iteration a non-convex projection over the constraint set $\mathcal{A}$.
With math terms, this leads to:
\begin{align*}
x_{i+1} = \Pi_{k, \mathcal{A}} \left(x_i - \mu_i \nabla f(x_i) \right),~\text{where } \mu_i \in \R.
\end{align*} 
%where $\mu_i \in \R$ is the step size.
% and $\Pi_{k, \mathcal{A}}(\cdot)$ is the projection operator over

While the merits of plain IHT, as described above, are widely known for simple sets $\mathcal{A}$ and specific functions $f$ (\emph{cf.}, \cite{blumensath2009iterative, bahmani2013greedy, jain2014iterative, barber2017gradient}), momentum-based acceleration techniques in IHT have not received significant attention in more generic ML settings.
Here, we study a simple momentum-variant of IHT, previously proposed in \cite{kyrillidis2011recipes, kyrillidis2014matrix}, that satisfies the following recursions:
\begin{align}
x_{i+1} = \Pi_{k, \mathcal{A}} \left(u_i - \mu_i \nabla_{\mathcal{T}_i} f(u_i) \right), 
\nonumber
\end{align}
and 
\begin{align}
u_{i+1} = x_{i+1} + \tau \cdot (x_{i+1} - x_i). \label{eq:main_rec}
\end{align} 
Here, $\nabla_{\mathcal{T}_i} f(\cdot)$ denotes restriction of the gradient on the subspace spanned by set $\mathcal{T}$; more details below.
$\tau$ is the momentum step size, used to properly weight previous estimates with the current one, based on \cite{nesterov2013introductory}.\footnote{Nesterov's acceleration is an improved version of Polyak's classical momentum \cite{polyak1964some} schemes. Understanding when and how hard thresholding operations still work for the whole family of momentum algorithms is open for future research direction.}
Despite the simplicity of \eqref{eq:main_rec}, to the best of our knowledge, there are no convergence guarantees for generic $f$, neither any characterization of its performance w.r.t. $\tau$ values. 
Nevertheless, its superior performance has been observed under various settings and configurations \cite{kyrillidis2011recipes, blumensath2012accelerated, kyrillidis2014matrix}.

\emph{In this paper, we study this accelerated IHT variant, as described in Algorithm \ref{algo:alps}.}
This algorithm was originally presented in \cite{kyrillidis2011recipes, kyrillidis2014matrix}. %, where its empirical performance was found superior compared to state of the art, for the case of sparse linear regression.
%However, \cite{kyrillidis2011recipes, kyrillidis2014matrix} covers only a special case  (\emph{i.e.}, squared loss) of our setting, and the theory there is restricted, and further needs justification (\emph{e.g.}, the role of $\tau$ in the convergence behavior is not studied).
%Our aim is to extend such results beyond squared loss and prese
For simplicity, we will focus on the case of sparsity; same notions can be extended to more complicated sets $\mathcal{A}$.

Some notation first: given gradient $\nabla f(x) \in \R^\dim$, and given a subset of $[\dim]$, say $\mathcal{T} \subseteq [\dim]$, $\nabla_{\mathcal{T}} f(x) \in \R^\dim$ has entries from $\nabla f(x)$, only indexed by $\mathcal{T}$.\footnote{Here, we abuse a bit the notation for the case of low rank structure $\mathcal{A}$: in that case $\nabla_{\mathcal{T}} f(x) \in \R^{m \times n}$ denotes the part of $\nabla f(x)$ that ``lives'' in the subspace spanned by the atoms in $\mathcal{T}$.}
$\mathcal{T}^c$ represents the complement of $[\dim] \setminus \mathcal{T}$.
%Given model $\mathcal{A}$, the index of atoms that constitute a ``signal'' $x$ is provided by the support function $\texttt{supp}(x)$; \emph{e.g.}, in the plain sparsity model, $\texttt{supp}(x)$ returns the indices of the nonzero entries of $x$. 
%More discussion regarding models $\mathcal{A}$ is found in Subsection \ref{subsec:model}.

\begin{algorithm}[!ht]
   \caption{Accelerated IHT algorithm}\label{algo:alps}
\begin{algorithmic}[1]
   \STATE {\bfseries Input:} Tolerance $ \eta $, $T$, $\kappa > 1$, $\mu_i > 0$, model $\mathcal{A}$, $k \in \mathbb{Z}_+$. 
   \STATE {\bfseries Initialize:} $ x_0, u_0 \leftarrow 0$, $ \mathcal{U}_0 \leftarrow \lbrace \emptyset \rbrace $. Set $\xi = \tfrac{2(\kappa - 1)}{\kappa + 1}$; select $\tau$ s.t. $|\tau| \leq \tfrac{1 - \varphi \xi^{1/2}}{\varphi \xi^{1/2}}$, where $\varphi = \tfrac{1 + \sqrt{5}}{2}$.%\\
%   \hrulefill
   \STATE {\bfseries repeat} %\REPEAT
   \STATE \hspace{0.16cm} $ \mathcal{T}_i \leftarrow \texttt{supp}_{\mathcal{A}}\left(\Pi_{k, \mathcal{A}} \left(\nabla_{\mathcal{U}_i^c} f(u_i) \right)\right) \cup \mathcal{U}_i$ %\hfill \textit{(Active support expansion)}
   \STATE \hspace{0.16cm} $ \bar{u}_i = u_i  - \mu_i \nabla_{\mathcal{T}_i} f(u_i) $ %\hfill \textit{(Gradient descent step)}
   \STATE \hspace{0.16cm} $ x_{i+1} = \Pi_{k, \mathcal{A}}\left(\bar{u}_i\right)^\dagger $ %\hfill \textit{(Hard-thresholding projection)}
   %\STATE \hspace{0.16cm} \textbf{OPTIONAL}: Debias step on $x_{i+1}$, restricted on the support $\texttt{supp}(x_{i+1}) $.
   %\STATE \hspace{0.16cm} Select a constant $\tau_{i+1}$ or an adaptive step size \hfill \textit{(Memory step size selection)}
   \STATE \hspace{0.16cm} $ u_{i+1} = x_{i+1} + \tau \left(x_{i+1} - x_i\right)$ where $ \mathcal{U}_{i+1} \leftarrow \texttt{supp}_{\mathcal{A}}(u_{i+1}) $ %\hfill \textit{(Memory momentum update)}
   %\STATE $ i \leftarrow i + 1 $.
   \STATE {\bfseries until} $\vectornorm{x_i - x_{i-1}} \leq \eta \vectornorm{x_i} $ or after $T$ iterations. %\\
   %\hrulefill 
   \STATE $^\dagger$\textit{Optional}: Debias step on $x_{i+1}$, restricted on the support $\texttt{supp}_{\mathcal{A}}(x_{i+1}) $.
\end{algorithmic}
\end{algorithm}

Algorithm \ref{algo:alps} maintains and updates an estimate of the optimum at every iteration. 
It does so by maintaining two sequences of variables: $x_i$'s that represent our putative estimates per iteration, and $u_i$'s that model the effect of ``friction'' (memory) in the iterates.
The first step in each iteration is \emph{active support expansion}: we expand support set $\mathcal{U}_i$ of $u_i$, by finding the indices of $k$ atoms of the largest entries in the gradient in the complement of $\mathcal{U}_i$.
This step results into set $\mathcal{T}_i$ and makes sure that per iteration we enrich the active support by ``exploring'' enough outside of it.
The following two steps perform the recursion in \eqref{eq:main_rec}, restricted on $\mathcal{T}_i$; \emph{i.e.}, we perform a gradient step, followed by a projection onto $\mathcal{A}$; finally, we update the auxiliary sequence $u$ by using previous estimates as momentum.
The iterations terminate once certain condition holds.

Some observations: % regarding the motions in Algorithm \ref{algo:alps}:
Set $\mathcal{T}_i$ has cardinality at most $3k$;
$x_i$ estimates are always $k$-sparse;
intermediate ``signal'' $u_i$ has cardinality at most $2k$, as the superposition of two $k$-sparse ``signals''.

%% file: theory.tex
%!TEX root = AccIHT.tex

%\section{Theoretical guarantees}{\label{sec:theory}}
\section{Theoretical study}{\label{sec:theory}}
%Here, we present the main theoretical contribution of this paper: proof of convergence for Algorithm \ref{algo:alps}, and its convergence rate.
Our study\footnote{Our focus is to study optimization guarantees (convergence), not statistical ones (required number of measurements, etc). Our aim is the investigation of accelerated IHT and under which conditions it leads to convergence; not its one-to-one comparison with plain IHT schemes.} starts with the description of the dynamics involved per iteration (Lemma \ref{thm:iter_inv}), followed by the conditions and eligible parameters that lead to convergence.
Proofs are deferred to the Appendix.

\begin{lemma}[Iteration invariant]{\label{thm:iter_inv}}
Consider the non-convex optimization problem in \eqref{eq:problem}, for given structure $\mathcal{A}$, associated with $\Pi_{k, \mathcal{A}}(\cdot)$, and loss function $f(x) := \tfrac{1}{2}\|b - \Phi x\|_2^2$, satisfying restricted strong convexity and smoothness properties, through generalized RIP, over $3k$ sparse ``signals'', with parameters $\alpha_{3k}$ and $\beta_{3k}$, respectively.
Let $x^\star$ be the minimizer of $f$, with $\|x^\star\|_{0, \mathcal{A}} = k$. % and $f(x^\star) \leq f(y)$, for any $y \in \R^\dim$ such that $\|y\|_{0, \mathcal{A}} \leq 3k$.
Assuming $x_0 = 0$, Algorithm \ref{algo:alps} with $\mu_i = \tfrac{2}{\alpha_{3k} + \beta_{3k}}$ satisfies  $\forall \tau$ the following linear system at the $i$-th iteration:
\begin{align*}
\begin{bmatrix}
\|x_{i+1} - x^\star\|_2 \\ \|x_i - x^\star\|_2
\end{bmatrix} &\leq \underbrace{\begin{bmatrix}
\tfrac{2(\kappa - 1)}{\kappa + 1} \cdot |1 + \tau| & \tfrac{2(\kappa - 1)}{\kappa + 1} \cdot  |\tau| \\
1 & 0
\end{bmatrix}}_{:=A} \cdot 
\begin{bmatrix}
\|x_i - x^\star\|_2 \\
\|x_{i-1} - x^\star\|_2
\end{bmatrix} + 
\begin{bmatrix}
1 \\ 0
\end{bmatrix} \tfrac{2 \sqrt{\beta_{2k}}}{\alpha_{3k} + \beta_{3k}} \left\|\varepsilon \right\|_2
\end{align*}
where $\kappa := \tfrac{\beta_{3k}}{\alpha_{3k}}$, and $b = \Phi x^\star + \varepsilon$.
\end{lemma}

\noindent \textit{Proof ideas involved:}
The proof is composed mostly of algebraic manipulations. 
For exact projection $\Pi_{k, \mathcal{A}}(\cdot)$ and due to the optimality of the step $x_{i+1} = \Pi_{k, \mathcal{A}}\left(\bar{u}_i\right)$, we observe that $\|x_{i+1} - x^\star\|_2^2 \leq 2 \left \langle x_{i+1} - x^\star, ~\bar{u}_i - x^\star \right \rangle$.
Using Definitions \ref{def_00}-\ref{def_01}, %we prove a version of Lemma 2 in \cite{agarwal2010fast} over non-convex constraint sets, using optimality conditions over low-dimensional structures \cite{beck2015minimization}. 
and due to the restriction of the active subspace to the set $\mathcal{T}_i$ per iteration (most operations --\emph{i.e.}, inner products, Euclidean distances, etc-- involved in the proof are applied on ``signals'' comprised of at most $3k$ atoms)
we get the two-step recursion.
%\begin{align*}
%\|x_{i+1} - x^\star\|_2 &\leq \left(1 - \tfrac{\alpha}{\beta}\right) \cdot |1 + \tau| \cdot \|x_i - x^\star\|_2 \\
%							   & \quad \quad \quad \quad+ \left(1 - \tfrac{\alpha}{\beta}\right) \cdot |\tau| \cdot \|x_{i-1} - x^\star\|_2.
%\end{align*}
%Finally, we convert this second-order linear system into a two-dimensional first-order system, that produces the desired recursion. 
See Appendix \ref{sec:proofs} for a detailed proof.

Lemma \ref{thm:iter_inv} just states the iteration invariant of Algorithm \ref{algo:alps}; it does not guarantee convergence.
To do so, we need to state some interesting properties of $A$. The proof is elementary and is omitted. 

%\begin{lemma}{\label{lemma:nth-power0}}
%Let $A$ be the $2 \times 2$ matrix, as defined above, parameterized by constants $0 < \alpha < \beta$, and user-defined parameter $\tau$. 
%Denote $\xi := 1 - \sfrac{\alpha}{\beta}$ and let $\imath$ represent the imaginary unit such that $\imath^2 = -1$.
%Depending on the values of $\tau$, the eigenvalues of $A$ satisfy the following rules: \vspace{-0.2cm}
%\begin{itemize}[leftmargin = 0.5cm]
%\item When $\tau < -\left(1 + \sfrac{2}{\xi}\right) - 2\sqrt{\sfrac{1}{\xi}\cdot \left(\sfrac{1}{\xi} + 1\right)}$ ~or~ $\tau > -\left(1 + \sfrac{2}{\xi}\right) + 2\sqrt{\sfrac{1}{\xi}\cdot \left(\sfrac{1}{\xi} + 1\right)}$, there are two real and distinct eigenvalues:
%$\lambda_{1,2} = \sfrac{\xi(1+\tau)}{2} \pm \sqrt{\sfrac{\xi^2(1+\tau)^2}{4} + \xi\tau}$. \vspace{-0.3cm}
%\item When $\tau = -\left(1 + \sfrac{2}{\xi}\right) \pm 2\sqrt{\sfrac{1}{\xi}\cdot \left(\sfrac{1}{\xi} + 1\right)}$, then there is only one root:
%$\lambda = \sfrac{\xi(1+\tau)}{2}$. \vspace{-0.3cm}
%\item When $\tau \in \left(-\left(1 + \sfrac{2}{\xi}\right) - 2\sqrt{\sfrac{1}{\xi}\cdot \left(\sfrac{1}{\xi} + 1\right)}, -\left(1 + \sfrac{2}{\xi}\right) + 2\sqrt{\sfrac{1}{\xi}\cdot \left(\sfrac{1}{\xi} + 1\right)}\right)$, there are two \emph{complex} roots:
%$\lambda_{1,2} = \sfrac{\xi(1+\tau)}{2} \pm \imath \cdot \sqrt{-\sfrac{\xi^2(1+\tau)^2}{4} - \xi\tau}$.
%\end{itemize}
%\end{lemma}

\begin{lemma}{\label{lemma:nth-power0}}
Let $A$ be the $2 \times 2$ matrix, as defined above, parameterized by $ \kappa > 1$, and user-defined parameter $\tau$. 
Denote $\xi := \tfrac{2(\kappa - 1)}{\kappa + 1}$. % and let $\imath$ represent the imaginary unit such that $\imath^2 = -1$.
The characteristic polynomial of $A$ is defined as:
\begin{align*}
\lambda^2 - \textup{\text{Tr}}(A) \cdot \lambda + \textup{\text{det}}(A) = 0
\end{align*}
where $\lambda$ represent the eigenvalue(s) of $A$.
Define $\Delta := \textup{\text{Tr}}(A)^2 - 4 \cdot \textup{\text{det}}(A) = \xi^2 \cdot (1 + \tau)^2 + 4 \xi \cdot |\tau|$.
Then, the eigenvalues of $A$ satisfy the expression: $\lambda = \tfrac{\xi \cdot |1 + \tau| \pm \sqrt{\Delta}}{2}$.
Depending on the values of $\alpha, \beta, \tau$:
\begin{itemize}[leftmargin = 0.5cm]
\item $A$ has a unique eigenvalue $\lambda  = \tfrac{\xi \cdot |1 + \tau|}{2}$, if $\Delta = 0$. This happens when $\kappa = 1$ and is not considered in this paper (we assume functions $f$ with curvature). 
\item $A$ has two complex eigenvalues; this happens when $\Delta < 0$. By construction, this case does not happen in our scenaria, since $\kappa > 1$. 
\item For all other cases, $A$ has two distinct real eigenvalues, satisfying $\lambda_{1,2} = \tfrac{\xi \cdot |1 + \tau|}{2} \pm \tfrac{\sqrt{\xi^2 \cdot (1 + \tau) + 4\xi \cdot |\tau|}}{2}$.
\end{itemize}
\end{lemma}

Define $y(i+1) = \begin{bmatrix} \|x_{i+1} - x^\star\|_2 \\ \|x_i - x^\star\|_2 \end{bmatrix} $; then, the linear system in Lemma \ref{thm:iter_inv} for the $i$-th iteration becomes $y(i+1) \leq A \cdot y(i)$.
$A$ has only non-negative values; we can unfold this linear system over $T$ iterations such that 
\begin{align*}
y(T) \leq A^{T} \cdot y(0) + \left(\sum_{i = 0}^{T-1} A^i \right) \cdot \begin{bmatrix} 1 \\ 0 \end{bmatrix} \cdot \tfrac{2 \sqrt{\beta_{2k}}}{\alpha_{3k} + \beta_{3k}} \left\|\varepsilon \right\|_2.
\end{align*}
Here, we make the convention that $x_{-1} = x_{0} = 0$, such that $y(0) = \begin{bmatrix} \|x_{0} - x^\star\|_2 \\ \|x_{-1} - x^\star\|_2 \end{bmatrix} = \begin{bmatrix} 1 \\ 1 \end{bmatrix} \cdot \|x^\star\|_2$.
The following lemma describes how one can compute a power of a $2 \times 2$ matrix $A$, $A^i$, through the eigenvalues $\lambda_{1, 2}$ (real and distinct eigenvalues); the proof is provided in Section \ref{subsec:n-thpowerx}. 
To the best of our knowledge, there is no detailed proof on this lemma in the literature. 
\begin{lemma}[\cite{williams1992the}]{\label{lemma:nth-power2}}
Let $A$ be a $2 \times 2$ matrix with real eigenvalues $\lambda_{1,2}$. 
Then, the following expression holds, when $\lambda_1 \neq \lambda_2$:
\begin{align*}
A^i = \frac{\lambda_1^i - \lambda_2^i}{\lambda_1 - \lambda_2} \cdot A - \lambda_1 \lambda_2 \cdot \frac{\lambda_1^{i-1} - \lambda_2^{i-1}}{\lambda_1 - \lambda_2} \cdot I
\end{align*} 
where $\lambda_i$ denotes the $i$-th eigenvalue of A in order.
%and when $\lambda_1 = \lambda_2 \equiv \lambda$:
%\begin{align*}
%A^i = \lambda^i \cdot I + i \cdot \lambda^{i-1} \cdot (A - \lambda \cdot I)
%\end{align*}
\end{lemma}

Then, the main recursion takes the following form:
\begin{small}
\begin{align}
y(T) \leq \frac{\lambda_1^T - \lambda_2^T}{\lambda_1 - \lambda_2} \cdot A \cdot y(0) - \lambda_1 \lambda_2 \frac{\lambda_1^{T-1} - \lambda_2^{T-1}}{\lambda_1 - \lambda_2} \cdot y(0) + \left(\sum_{i = 0}^{T-1} A^i \right) \cdot \begin{bmatrix} 1 \\ 0 \end{bmatrix} \cdot \tfrac{2 \sqrt{\beta_{2k}}}{\alpha_{3k} + \beta_{3k}} \left\|\varepsilon \right\|_2. \label{eq:main_iteration}
\end{align}
\end{small}

Observe that, in order to achieve convergence (\emph{i.e.}, the RHS convergences to zero), eigenvalues play a crucial role:
Both $A$ and $y(0)$ are constant quantities, and only how fast the quantities $\lambda_1^T - \lambda_2^T$ and $\lambda_1^{T-1} - \lambda_2^{T-1}$ ``shrink'' matter most.
%In the sequel, we will focus on the more plausible case where $\lambda_1 \neq \lambda_2$.

Given that eigenvalues appear in the above expressions in some power (\emph{i.e.}, $\lambda_{1,2}^T$ and $\lambda_{1,2}^{T-1}$), we require $|\lambda_{1,2}| < 1$ for convergence. %this results into assumptions on $\alpha, \beta$ and $\tau$, similar to the case in the main text.
%Such study are missing from previous work \cite{kyrillidis2011recipes, kyrillidis2014matrix}.
%Then, a simple argument to guarantee convergence comes from the contraction mapping theorem \cite{kelley1999iterative}:
%\begin{align*}
%y(i+1) \leq A \cdot y(i) ~\Longrightarrow~ \|y(i+1)\|_2 \leq \| A \cdot y(i)\|_2 ~\Longrightarrow~ y(i+1) \leq \|A\|_{2 \rightarrow 2} \cdot y(i),
%\end{align*}
%where $\|A\|_{2 \rightarrow 2} = \max_i |\lambda_i|$ is the operator norm over $A$.
%
%Thus, as long as $\|A\|_{2 \rightarrow 2} \leq 1$, we can then unfold the above recursion, as in Lemma \ref{thm:iter_inv}, to show convergence for our algorithm, with contraction constant $\|A\|_{2 \rightarrow 2}^i$.
%\emph{I.e.}, in order to obtain $\varepsilon > 0$ approximate solution, such that $\|x_i - x^\star\|_2 \leq \varepsilon$, Algorithm \ref{algo:alps} requires $O\left(\sfrac{\left(\log \tfrac{\|x^\star\|_2}{\varepsilon}\right)}{\left(-\log \|A\|_{2 \rightarrow 2}\right)}\right)$ iterations (linear convergence rate).
%
%We need to resolve the conditions that lead to $\|A\|_{2 \rightarrow 2} \leq 1$. 
%To do so, we will study $\tau$ selections accordingly, in order to cover all three subcases presented in Lemma \ref{lemma:nth-power0}. 
%Studies on the selection $\tau$ parameter is missing from previous works \cite{kyrillidis2011recipes, kyrillidis2014matrix}.
%\vspace{-0.2cm}
%Let $\xi = 1 - \tfrac{\alpha}{\beta} < 1$.
%\begin{itemize}[leftmargin = 0.5cm]
%
To achieve $|\lambda_{1,2}| < 1$, we have:
%\item When $A$ has two \emph{distinct and real} eigenvalues, $|\lambda_{1,2}| \leq 1$ translates into the following. 
\begin{align*}
|\lambda_{1,2}| &= \left|\tfrac{\xi \cdot |1+\tau|}{2} \pm \sqrt{\tfrac{\xi^2(1+\tau)^2}{4} + \xi \cdot |\tau|}\right| \\
&\leq \left|\tfrac{\xi \cdot |1+\tau|}{2}\right| + \left|\sqrt{\tfrac{\xi^2(1+\tau)^2}{4} + \xi \cdot |\tau|}\right| \\
& \stackrel{(i)}{\leq}  \tfrac{\xi \cdot |1+\tau|}{2} + \tfrac{1}{2} \sqrt{ \xi(1+|\tau|)^2 + 4\xi(1+|\tau|)^2} \\
&\stackrel{(i)}{\leq}  \tfrac{\xi^{\frac{1}{2}}(1+|\tau|)}{2} + \tfrac{\sqrt{5}}{2}  \xi^{\frac{1}{2}}(1+|\tau|) \\
& = \varphi \cdot \xi^{\frac{1}{2}}(1+|\tau|)
\end{align*}
where $(i)$ is due to $\xi < 1$, and $\varphi = \sfrac{(1 + \sqrt{5})}{2}$ denotes the golden ratio.
Thus, upper bounding the RHS to ensure $|\lambda_{1,2}| < 1$ implies $|\tau| < \tfrac{1 - \varphi \cdot \xi^{1/2}}{\varphi \cdot \xi^{1/2}}$. 
Also, under this requirement, $\sum_{i = 0}^{T-1} A^i$ converges to: 
\begin{align*}
\sum_{i = 0}^{T-1} A^i = (I - A)^{-1}(I - A^T),
\end{align*}
where:
\begin{align*}
B:= (I - A)^{-1} = \tfrac{1}{1 - \xi(1 + 2\tau)} \begin{bmatrix} 1 & \xi \tau \\ 1 & 1 - \xi(1 + \tau) \end{bmatrix}.
\end{align*}

Using the assumption $|\lambda_{1,2}| < 1$ for $|\tau| < \tfrac{1 - \varphi \cdot \xi^{1/2}}{\varphi \cdot \xi^{1/2}}$, \eqref{eq:main_iteration} further transforms to:
\begin{align*}
y(T) &\leq \frac{\lambda_1^T - \lambda_2^T}{\lambda_1 - \lambda_2} \cdot A \cdot y(0) - \lambda_1 \lambda_2 \frac{\lambda_1^{T-1} - \lambda_2^{T-1}}{\lambda_1 - \lambda_2} \cdot y(0) + B (I - A^{T}) \begin{bmatrix} 1 \\ 0 \end{bmatrix} \tfrac{2 \sqrt{\beta_{2k}}}{\alpha_{3k} + \beta_{3k}} \left\|\varepsilon \right\|_2 \\
         &\stackrel{(i)}{\leq} \frac{|\lambda_1|^T + |\lambda_2|^T}{|\lambda_1| - |\lambda_2|} \cdot A \cdot y(0) + |\lambda_1 \lambda_2| \cdot \frac{|\lambda_1|^{T-1} + |\lambda_2|^{T-1}}{|\lambda_1| - |\lambda_2|} \cdot y(0) + B (I - A^{T}) \begin{bmatrix} 1 \\ 0 \end{bmatrix} \tfrac{2 \sqrt{\beta_{2k}}}{\alpha_{3k} + \beta_{3k}} \left\|\varepsilon \right\|_2 \\
         &\stackrel{(ii)}{\leq} \frac{2 |\lambda_1|^T}{|\lambda_1| - |\lambda_2|}\cdot \left(A  + I \right) \cdot y(0)  + B (I - A^{T})\begin{bmatrix} 1 \\ 0 \end{bmatrix} \tfrac{2 \sqrt{\beta_{2k}}}{\alpha_{3k} + \beta_{3k}} \left\|\varepsilon \right\|_2
\end{align*}
where $(i)$ is due to $A \cdot y(0)$ and $y(0)$ being positive quantities,  and $(ii)$ is due to $1 > |\lambda_1| > |\lambda_2|$.
Focusing on the error term, we can split and expand the term similarly to obtain:
\begin{align*}
B (I - A^{T})\begin{bmatrix} 1 \\ 0 \end{bmatrix} \tfrac{2 \sqrt{\beta_{2k}}}{\alpha_{3k} + \beta_{3k}} \left\|\varepsilon \right\|_2 
&\leq \frac{2 |\lambda_1|^T}{|\lambda_1| - |\lambda_2|} B (A + I) \begin{bmatrix} 1 \\ 0 \end{bmatrix} \tfrac{2 \sqrt{\beta_{2k}}}{\alpha_{3k} + \beta_{3k}} \left\|\varepsilon \right\|_2 + B \begin{bmatrix} 1 \\ 0 \end{bmatrix} \tfrac{2 \sqrt{\beta_{2k}}}{\alpha_{3k} + \beta_{3k}} \left\|\varepsilon \right\|_2 \\ 
&= \frac{2 |\lambda_1|^T}{|\lambda_1| - |\lambda_2|} \tfrac{1}{1 - \xi(1 + 2\tau)} \begin{bmatrix} 1 + \xi(1 + 2\tau) \\ 1 + \xi(1 + \tau) \end{bmatrix} \tfrac{2 \sqrt{\beta_{2k}}}{\alpha_{3k} + \beta_{3k}} \left\|\varepsilon \right\|_2 + \tfrac{1}{1 - \xi(1 + 2\tau)} \begin{bmatrix} 1 \\ 1 \end{bmatrix}\tfrac{2 \sqrt{\beta_{2k}}}{\alpha_{3k} + \beta_{3k}} \left\|\varepsilon \right\|_2 
\end{align*}

Combining the above, we get:
\begin{small}
\begin{align*}
y(T) &\leq \frac{2 |\lambda_1|^T}{|\lambda_1| - |\lambda_2|}\cdot \left(A  + I \right) \cdot y(0)  + \frac{2 |\lambda_1|^T}{|\lambda_1| - |\lambda_2|} \tfrac{1}{1 - \xi(1 + 2\tau)} \begin{bmatrix} 1 + \xi(1 + 2\tau) \\ 1 + \xi(1 + \tau) \end{bmatrix} \tfrac{2 \sqrt{\beta_{2k}}}{\alpha_{3k} + \beta_{3k}} \left\|\varepsilon \right\|_2 + \tfrac{1}{1 - \xi(1 + 2\tau)} \begin{bmatrix} 1 \\ 1 \end{bmatrix}\tfrac{2 \sqrt{\beta_{2k}}}{\alpha_{3k} + \beta_{3k}} \left\|\varepsilon \right\|_2 \\
&= \frac{2 |\lambda_1|^T}{|\lambda_1| - |\lambda_2|} \cdot \left( \left(A  + I \right) \cdot y(0) + \tfrac{1}{1 - \xi(1 + 2\tau)}  \begin{bmatrix} 1 + \xi(1 + 2\tau) \\ 1 + \xi(1 + \tau) \end{bmatrix} \tfrac{2 \sqrt{\beta_{2k}}}{\alpha_{3k} + \beta_{3k}} \left\|\varepsilon \right\|_2\right) + \tfrac{1}{1 - \xi(1 + 2\tau)} \begin{bmatrix} 1 \\ 1 \end{bmatrix}\tfrac{2 \sqrt{\beta_{2k}}}{\alpha_{3k} + \beta_{3k}} \left\|\varepsilon \right\|_2
\end{align*}
\end{small}
Focusing on the first entry of $y(T)$ and substituting the first row of $A$ and $y(0)$, we obtain the following inequality:
\begin{align}
\|x_{T} - x^\star\|_2 \leq \tfrac{2 \cdot |\lambda_1|^T}{|\lambda_1| - |\lambda_2|} \cdot \left( (1 + \xi(1 + 2\tau)) \cdot \|x^\star\|_2 + \tfrac{1 + \xi(1 + 2\tau)}{1 - \xi(1 + 2\tau)} \tfrac{2 \sqrt{\beta_{2k}}}{\alpha_{3k} + \beta_{3k}} \left\|\varepsilon \right\|_2\right) + \tfrac{1}{1 - \xi(1 + 2\tau)} \cdot \tfrac{2 \sqrt{\beta_{2k}}}{\alpha_{3k} + \beta_{3k}} \left\|\varepsilon \right\|_2. \label{eq:main_iteration2}
\end{align}
This suggests that, as long as $|\lambda_{1,2}| < 1$, the RHS ``shrinks'' exponentially with rate $|\lambda_1|^T$, but also depends (inverse proportionally) on the spectral gap $|\lambda_1| - |\lambda_2|$.
The above lead to the following convergence result for the noiseless case; similar results can be derived for noisy cases as well:
\begin{theorem}
%Consider the non-convex optimization problem in \eqref{eq:problem}, for given structure $\mathcal{A}$, associated with $\Pi_{k, \mathcal{A}}(\cdot)$, and loss function $f$, satisfying restricted strong convexity and smoothness properties over $4k$ sparse ``signals'', with parameters $\alpha$ and $\beta$, respectively.
Under the same assumptions with Lemma \ref{thm:iter_inv}, Algorithm \ref{algo:alps} returns a $\varepsilon$-approximate solution for $f(x) = \tfrac{1}{2} \|b - \Phi x\|_2^2$, such that $\|x_T - x^\star\|_2 \leq \zeta$, within $O\left(\log \tfrac{(1 + \xi \cdot (1 + 2\tau))}{\zeta \cdot (|\lambda_1| - |\lambda_2|)}\right)$ iterations (linear convergence rate).
\end{theorem}

\begin{proof}
We get this result by forcing the RHS of \eqref{eq:main_iteration2} be less than $\zeta > 0$. \emph{I.e.},
\begin{align}
\tfrac{2 \cdot |\lambda_1|^T}{|\lambda_1| - |\lambda_2|} \cdot (1 + \xi \cdot (1 + 2\tau)) \cdot \|x^\star\|_2 \leq \zeta &\Rightarrow \\
|\lambda_1|^T \leq \frac{\zeta \cdot (|\lambda_1| - |\lambda_2|)}{2 \cdot (1 + \xi \cdot (1 + 2\tau)) \cdot \|x^\star\|_2} &\Rightarrow \\
T \geq \left \lceil \frac{\log \frac{2 \cdot (1 + \xi \cdot (1 + 2\tau)) \cdot \|x^\star\|_2}{\zeta \cdot (|\lambda_1| - |\lambda_2|)}}{\log |\lambda_1|} \right \rceil
\end{align} This completes the proof.
\end{proof}

\begin{figure*}[t!]
\centering
\includegraphics[width=1\textwidth]{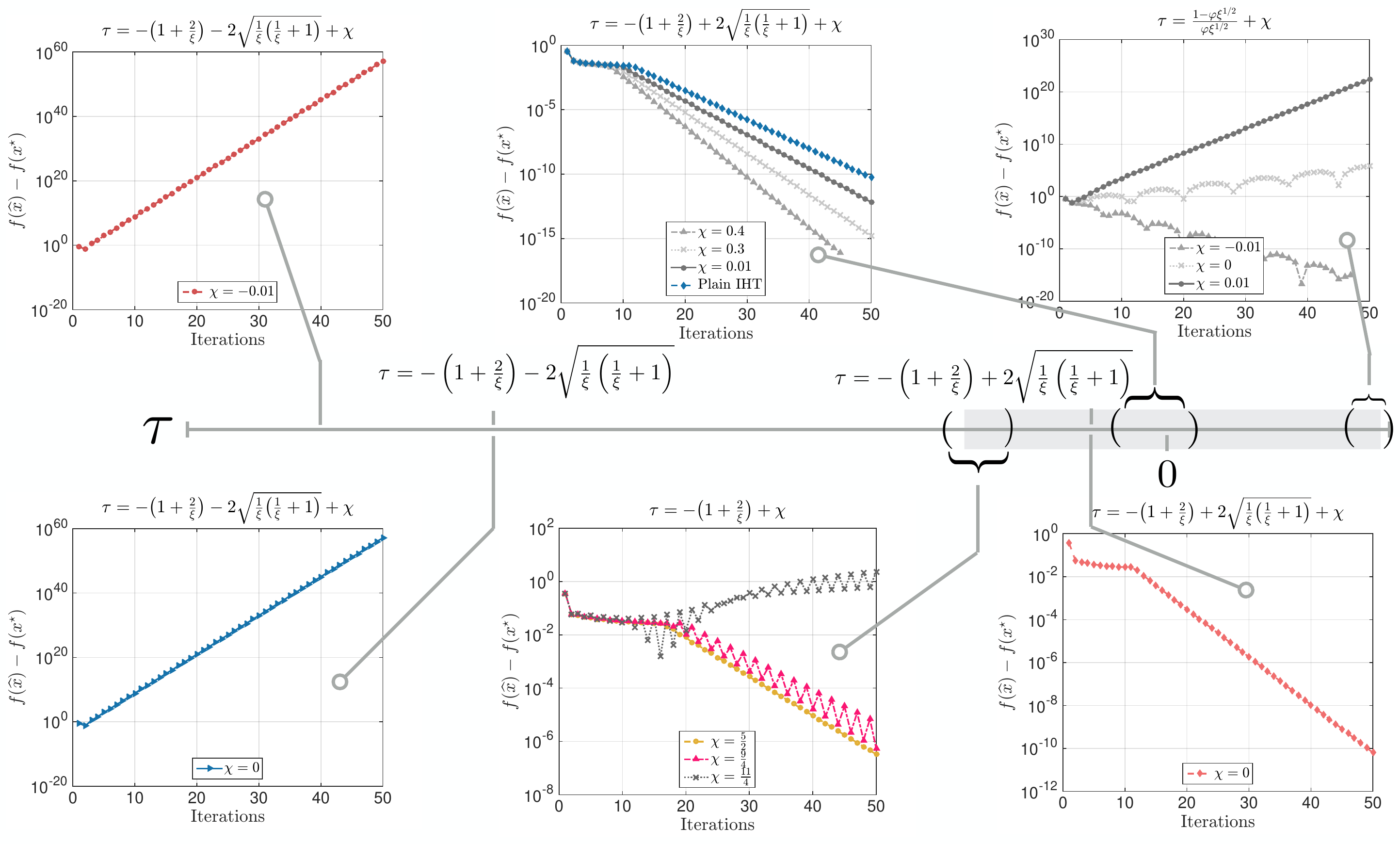}
\caption{Behavior of accelerated IHT method, applied on a toy example for sparse linear regression. Consider $\mathcal{A}$ as the plain sparsity model, and let $x^\star$ be a $k$-sparse ``signal'' in $\R^{10}$ for $k=2$, drawn from multivariate normal distribution. Also, $\|x^\star\|_2 = 1$. Let $b = \Phi x^\star$, with $\Phi \in \R^{6 \times 10}$ drawn entrywise i.i.d. from a normal distribution. Let $\mathcal{I}$ be an index set of $k$ columns in $\Phi$; there are ${n \choose k}$ possible such subsets. By Definitions \ref{def_00}-\ref{def_01}, we estimate $\alpha$ and $\beta$ as the $\lambda_{\min}(\Phi_{\mathcal{I}}^\top\Phi_{\mathcal{I}})$ and $\lambda_{\max}(\Phi_{\mathcal{I}}^\top \Phi_{\mathcal{I}})$, where $\Phi_{\mathcal{I}}$ is the submatrix of $\Phi$, indexed by $\mathcal{I}$. Here, $\alpha \approx 0.64$ and $\beta \approx 1.78$, which leads to $\xi = 1 - \sfrac{\alpha}{\beta} \approx 0.94$. We plot $f(\widehat{x}) - f(x^\star)$ vs. iteration count, where $f(x) = \tfrac{1}{2} \|b - \Phi x\|_2^2$. Gray shaded area on $\tau$ horizontal line corresponds to the range $|\tau| \leq (1 - \varphi \xi^{1/2})/(\varphi \xi^{1/2})$. (\textbf{Left panel, top and bottom row}). Accelerated IHT diverges for negative $\tau$, outside the $\tau$ shaded area. (\textbf{Middle panel, bottom row}). ``Rippling'' behavior for $\tau$ values close to the lower bound of converging $\tau$. (\textbf{Middle panel, top row}). Convergence behavior for accelerated IHT for various $\tau$ values and its comparison to plain IHT ($\tau = 0$). (\textbf{Right panel, top row}). Similar ``rippling'' behavior as $\tau$ approaches close to the upper bound of the shaded area; divergence is observed when $\tau$ goes beyond the shaded area (observe that, for $\tau$ values at the border of the shaded area, Algorithm \ref{algo:alps} still diverges and this is due to the approximation of $\xi$).}
\label{fig:momentum}
\end{figure*}

%% file: related.tex
%!TEX root = AccIHT.tex

\section{Related work}{\label{sec:related}}

Optimization schemes over low-dimensional structured models have a long history; due to lack of space, we refer the reader to \cite{kyrillidis2015structured} for an overview of discrete and convex approaches.
We note that there are both projected and proximal non-convex approaches that fit under our generic model, where no acceleration is assumed.
\emph{E.g.}, see \cite{bahmani2013greedy, yuan2014gradient, jain2014iterative, jain2016structured}; our present work fills this gap.
For non-convex proximal steps see \cite{gong2013general} and references therein; again no acceleration is considered.
Below, we focus on accelerated optimization variants, as well as Frank-Wolfe methods.

\emph{Related work on accelerated IHT variants.}
Accelerated IHT algorithms for sparse recovery were first introduced in \cite{qiu2010ecme, blumensath2012accelerated, kyrillidis2011recipes}. 
In \cite{qiu2010ecme}, the authors provide a \emph{double overrelaxation} thresholding scheme \cite{salakhutdinov2003adaptive} in order to accelerate their projected gradient descent variant for sparse linear regression; however, no theoretical guarantees are provided. 
In \cite{blumensath2012accelerated}, Blumensath accelerates standard IHT methods for the same problem \cite{blumensath2009iterative, garg2009gradient} using the double overrelaxation technique in \cite{qiu2010ecme}. 
His result contains theoretical proof of linear convergence, under the assumption that the overrelaxation step is used only when the objective function decreases.
However, this approach provides no guarantees that we might skip the acceleration term often, which leads back to the non-accelerated IHT version; see also \cite{salakhutdinov2003adaptive} for a similar approach on EM algorithms.
\cite{blanchard2015cgiht} describe a family of IHT variants, based on the conjugate gradient method \cite{hestenes1952methods}, that includes under its umbrella methods like in \cite{needell2009cosamp, foucart2011hard}, with the option to perform acceleration steps; however, no theoretical justification for convergence is provided when acceleration motions are used.
\cite{kyrillidis2011recipes, kyrillidis2014matrix} contain hard-thresholding variants, based on Nesterov's ideas \cite{nesterov2013introductory}; in \cite{kyrillidis2014matrix}, the authors provide convergence rate proofs for accelerated IHT when the objective is just least-squares; no generalization to convex $f$ is provided, neither a study on varied values of $\tau$. 
\cite{wei2015fast} includes a first attempt towards using adaptive $\tau$; his approach focuses on the least-squares objective, where a closed for solution for optimal $\tau$ is found \cite{kyrillidis2011recipes}. 
However, similar to \cite{blumensath2012accelerated}, it is not guaranteed whether and how often the momentum is used, neither how to set up $\tau$ in more generic objectives; see also Section \ref{sec:counter} in the appendix.
From a convex perspective, where the non-convex constraints are substituted by their convex relaxations (either in constrained or proximal setting), the work in \cite{bioucas2007new} and \cite{beck2009fast} is relevant to the current work: 
based on two-step methods for linear systems \cite{axelsson1996iterative}, \cite{bioucas2007new} extends these ideas to non-smooth (but convex) regularized linear systems, where $f$ is a least-squares term for image denoising purposes; see also \cite{beck2009fast}.
Similar to \cite{wei2015fast, blumensath2012accelerated}, \cite{bioucas2007new} considers variants of accelerated convex gradient descent that guarantee monotonic decrease of function values per iteration. % but with no guarantee of how often the momentum term is skipped in practice.

\emph{Related work on acceleration techniques.} 
Nesterov in \cite{nesterov1983method} was the first to consider acceleration techniques in convex optimization settings; see also Chapter 2.2 in \cite{nesterov2013introductory}.
Such acceleration schemes have been widely used as black box in machine learning and signal processing \cite{beck2009fast, bioucas2007new, schmidt2011convergence, shalev2014accelerated}.
\cite{nesterov2013gradient, becker2011templates} discuss restart heuristics, where momentum-related parameters are reset periodically after some iterations. %\footnote{This is equivalent with assuming $\alpha = 0$.}
\cite{o2015adaptive} provides some adaptive restart strategies, along with analysis and intuition on why they work in practice for simple convex problems.
%when the momentum term exceeds a critical value, one observes periodic ``riples" in the trace of the objective value, and the proposed strategies can recover the optimal linear convergence rate.
Acceleration in non-convex settings have been very recently considered in continuous settings \cite{ghadimi2016accelerated, carmon2016accelerated, agarwal2016finding}, where $f$ could be non-convex\footnote{The guarantees in these settings are restricted to finding a good stationary point.}. 
However, none of these works, beyond \cite{paquette2017catalyst}, consider non-convex and possibly discontinuous constraints---for instance the subset of $k$-sparse sets. % neither they come with practical implementations and applications.
In the case of \cite{paquette2017catalyst}, our work differs in that it explores better the low-dimensional constraint sets---however, we require $f$ to be convex.
More relevant to this work is \cite{li2015accelerated}: the authors consider non-convex proximal objective and apply ideas from \cite{beck2009fast} that lead to either monotone (skipping momentum steps) or nonmonotone function value decrease behavior; further, the theoretical guarantees are based on different tools than ours.
We identify that such research questions could be directions for future work.

\emph{Related work on dynamical systems and numerical analysis.}
Multi-step schemes originate from explicit finite differences discretization of dynamical systems; 
\emph{e.g.}, the so-called Heavy-ball method \cite{polyak1964some} origins from the discretization of the friction dynamical system
$\ddot{x}(t) + \gamma \dot{x}(t) + \nabla f(x(t)) = 0$, where $\gamma > 0$ plays the role of friction.
Recent developments on this subject can be found in \cite{wilson2016lyapunov}; see also references therein.
From a different perspective, Scieur et al. \cite{scieur2017integration} use multi-step methods from numerical analysis to discretize the gradient flow equation.
We believe that extending these ideas in non-convex domains (\emph{e.g.}, when non-convex constraints are included) is of potential interest for better understanding when and why momentum methods work in practical structured scenaria.

\emph{Related work on Frank-Wolfe variants:} The Frank-Wolfe (FW) algorithm~\cite{clarkson2010coresets, jaggi2013revisiting} is an iterative projection-free convex scheme for constrained minimization. 
%As opposed to projected or proximal gradient descent, 
Frank-Wolfe often has \emph{cheap} per iteration cost by solving a constrained linear program in each iteration.
The classical analysis by~\cite{clarkson2010coresets} presents sublinear convergence for general functions. 
For strongly convex functions, FW admits linear convergence if the optimum does not lie on the boundary of the constraint set; in that case, the algorithm still has sublinear convergence rate. 
To address the boundary issue, \cite{wolfe1970convergence} allows to move away from one of the already selected atoms, where linear convergence rate can be achieved \cite{lacoste2015global}. 
Similarly, the \emph{pairwise} variant introduced by~\cite{mitchell1974finding} also has a linear convergent rate. 
This variant adjusts the weights of two of already selected atoms.
\cite{garber2015faster} present a different perspective by showing linear convergence of classical FW over strongly convex sets and general functions. 
While several variants and sufficient conditions exist that admit linear convergence rates, the use of momentum for Frank-Wolfe, to the best of our knowledge is unexplored.

\begin{figure*}[t]
\centering
\includegraphics[width=0.44\textwidth, bb=0 0 500 500]{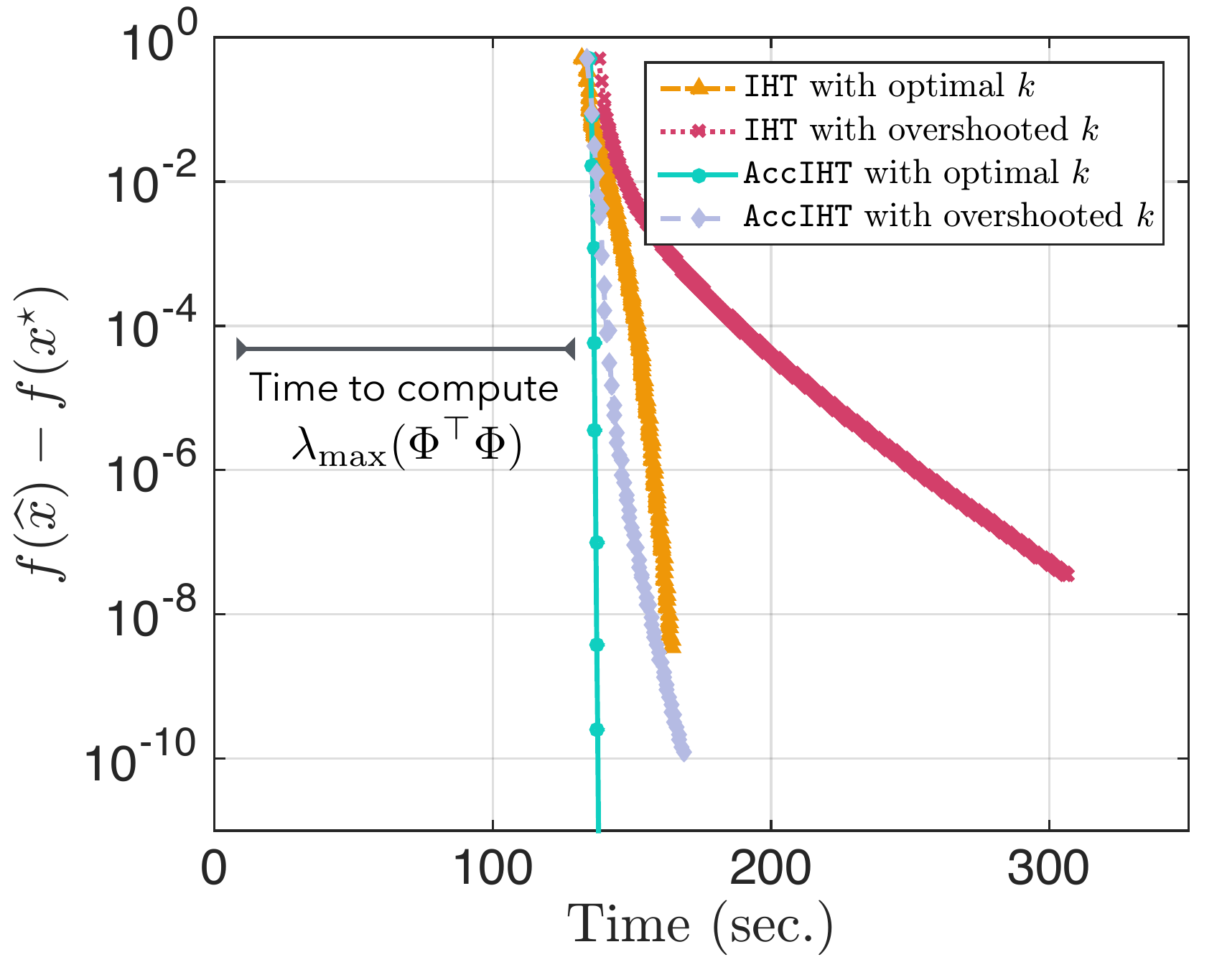} 
\includegraphics[width=0.49\textwidth, bb=0 0 550 550]{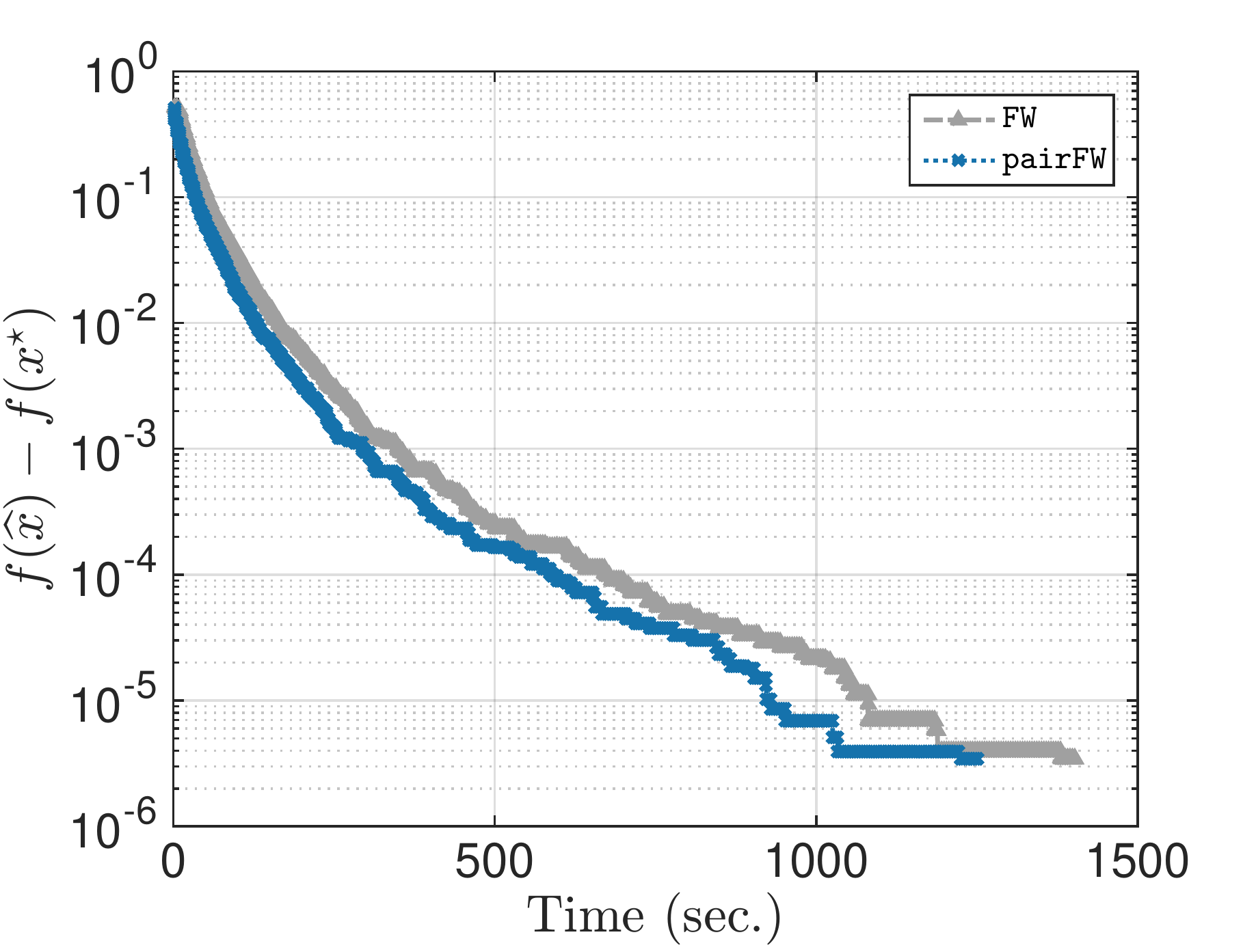} 
\caption{Time spent vs. function values gap $f(\widehat{x}) - f(x^\star)$. \texttt{AccIHT} corresponds to Algorithm \ref{algo:alps}. Beware in left plot the time spent to approximate step size, via computing $\lambda_{\max}(\Phi^\top \Phi)$.}
\label{fig:linreg1}
\end{figure*}

%% file: experiments.tex
%!TEX root = AccIHT.tex

\begin{figure*}[t]
\centering
\includegraphics[width=0.3\textwidth, bb=0 0 500 500]{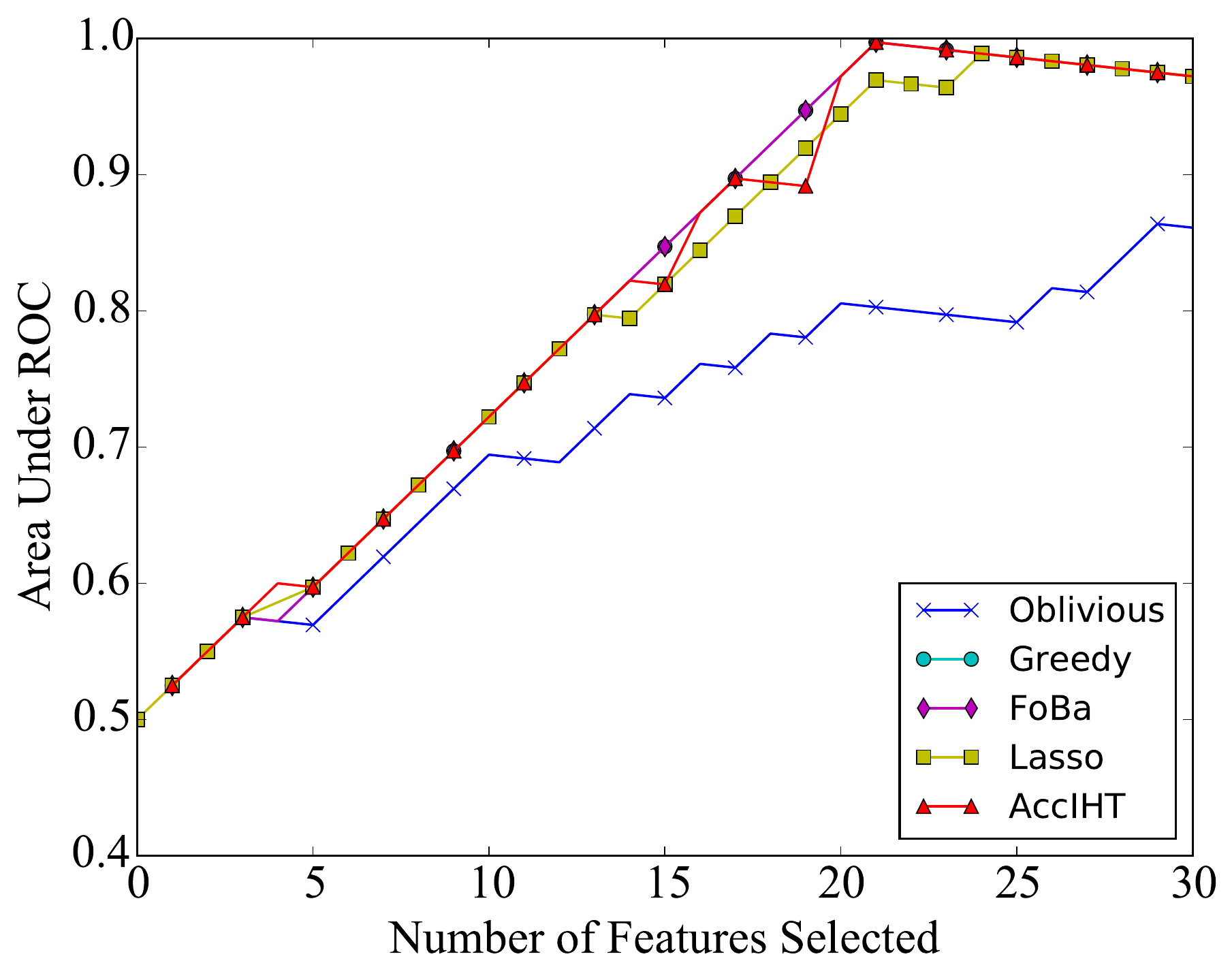} 
\includegraphics[width=0.34\textwidth, bb=0 0 550 550]{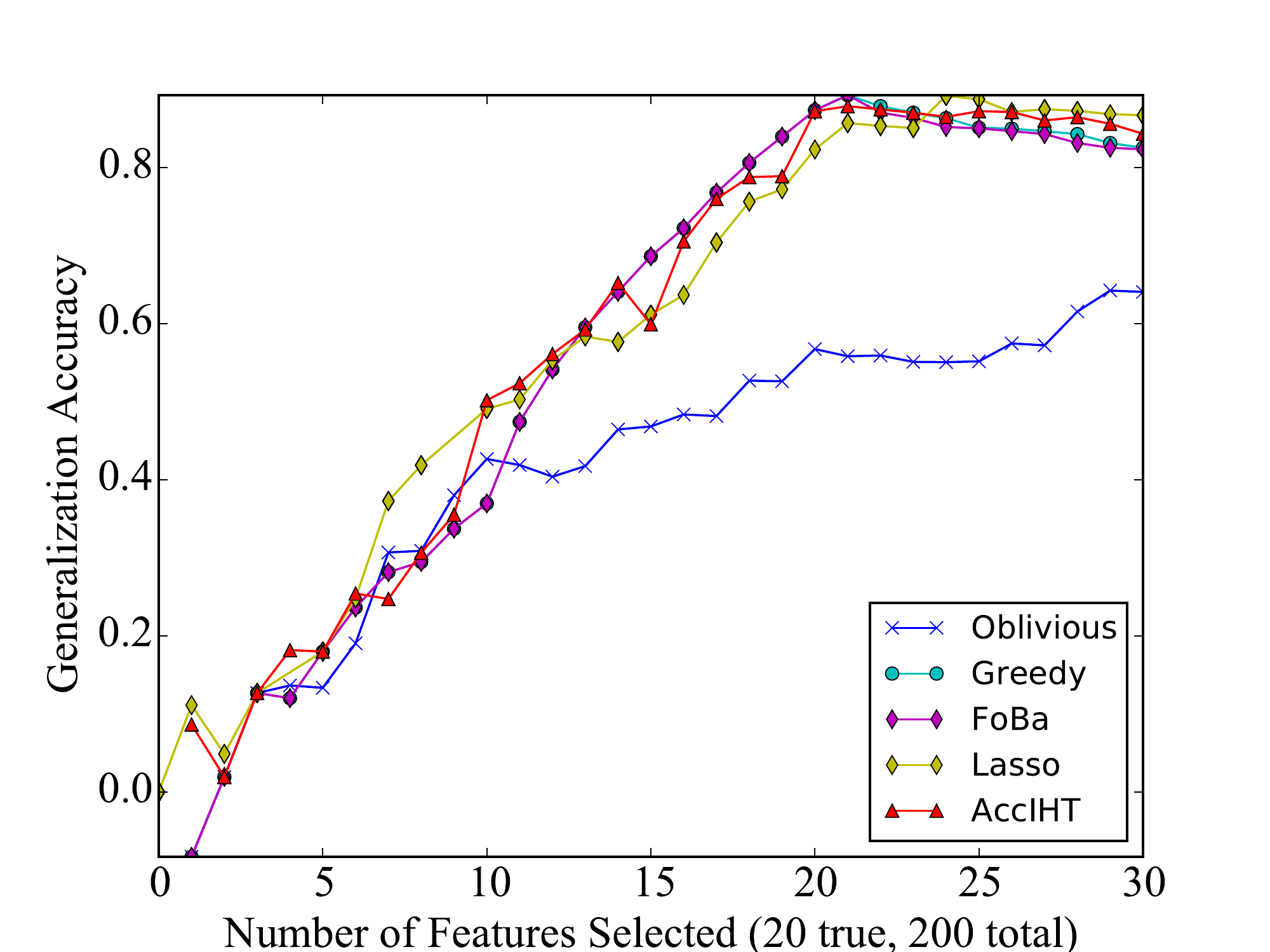}
\includegraphics[width=0.34\textwidth, bb=0 0 550 550]{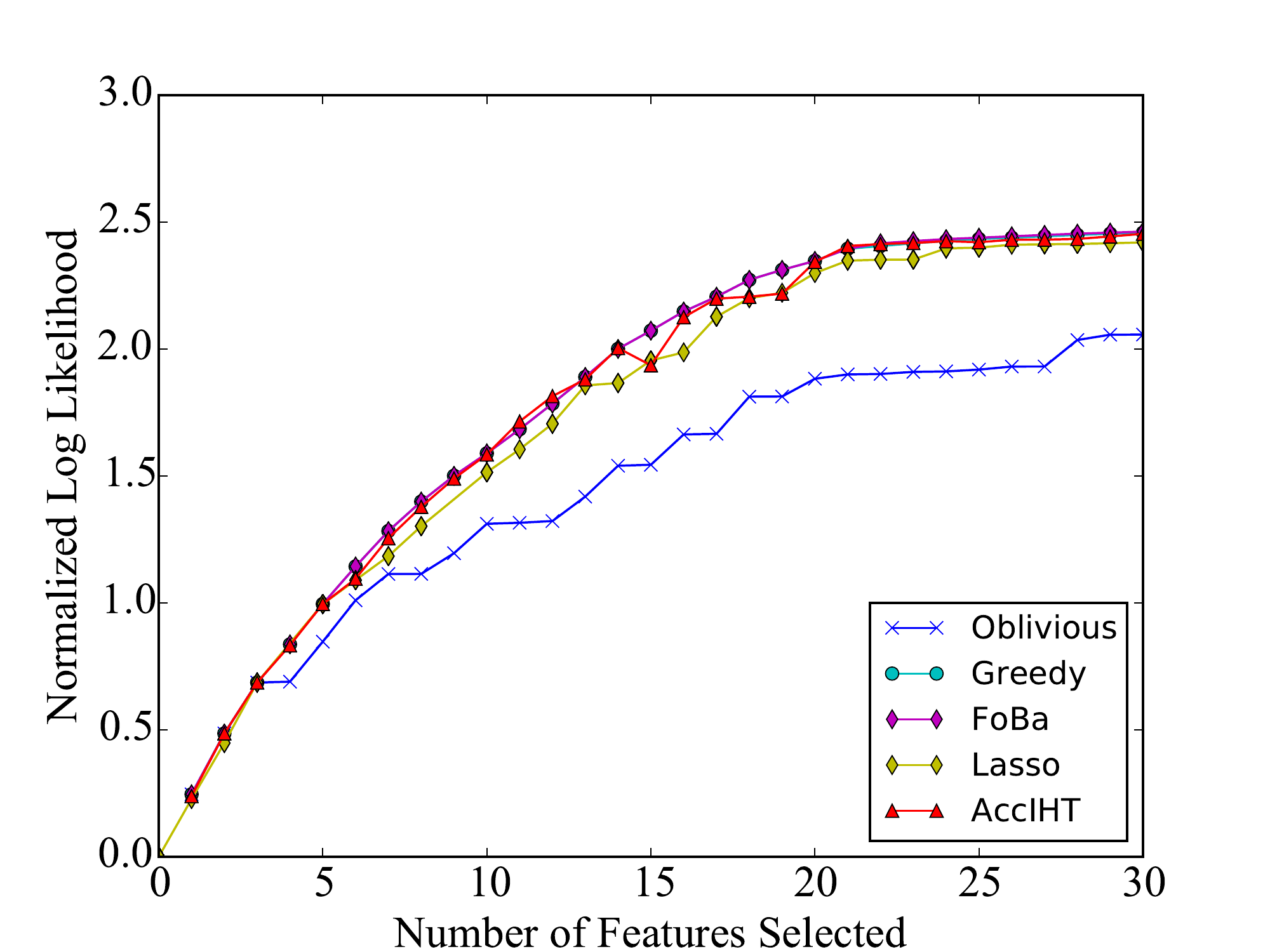}
\caption{Empirical evaluation. Algorithm \ref{algo:alps} achieves strong performance on true support recovery (left), generalization on test data (middle), and on training data fit (right) }
\label{fig:linreg2}
\end{figure*}

\section{Experiments}\label{sec:experiments}

We conducted simulations for different problems to verify our predictions.
In all experiments, we use constant $\tau = \sfrac{1}{4}$ as a potential universal momentum parameter.
Our experiments are proof of concept and demonstrate that accelerated projected gradient descent over non-convex structured sets can, not only offer high-quality recovery in practical settings, but offer much more scalable routines, compared to state-of-the-art.
\emph{Here, we present only a subset of our experimental findings and we encourage readers to go over the experimental results in the Appendix \ref{sec:moreexperiments}.}

\subsection{Sparse linear regression setting}{\label{subsec:linreg}}

For sparse linear regression, next we consider two simulated problems settings: $(i)$ with i.i.d. regressors, and $(ii)$ with correlated regressors.

\textit{Sparse linear regression under the i.i.d. Gaussian setting:} 
In this case, we consider a similar problem setting with \cite{lacoste2015global}, where $x^\star \in \R^\dim$ is the unknown normalized $k$-sparse vector, observed through the underdetermined set of linear equations: $b = \Phi x^\star$. 
$\Phi \in \R^{m \times \dim}$ is drawn randomly from a normal distribution.
We consider the standard least squares objective $f(x) = \tfrac{1}{2} \|b - \Phi x\|_2^2$ and the plain sparsity model $\mathcal{A}$ where $\|x\|_{0, \mathcal{A}} \equiv \|x\|_0$.

We compare Algorithm \ref{algo:alps} (abbreviated as AccIHT in plots) with two FW variants (see \cite{lacoste2015global} for \texttt{FW} and \texttt{pairFW})\footnote{Phase transition results and comparisons to standard algorithms such as CoSaMP (restricted to squared loss) can be found in \cite{kyrillidis2011recipes}, and thus omitted.}
Further, according to \cite{lacoste2015global}, \texttt{pairFW} performs better than \texttt{awayFW} for this problem case, as well as the plain IHT algorithm.
In this experiment, we use step sizes $1/\widehat{\beta}$ for Algorithm \ref{algo:alps} and IHT, where $\widehat{\beta} = \lambda_{\max}(\Phi^\top \Phi)$, which is suboptimal compared to the step size theory dictates. 
For the FW variants, we follow the setup in \cite{lacoste2015global} and set as the constraint set the $\lambda$-scaled $\ell_1$-norm ball, where $\lambda = 40$.
In the FW code, we further ``debias" the selected set of atoms per iteration, by performing fully corrective steps over putative solution (\emph{i.e.}, solve least-squares objective restricted over the active atom set) \cite{kyrillidis2011recipes, jain2016structured}.
We observed that such steps are necessary in our setting, in order FW to be competitive with Algorithm \ref{algo:alps} and IHT.
For IHT and Algorithm \ref{algo:alps}, we set input $k$ either exactly or use the $\ell_1$-norm phase transition plots \cite{donoho2005neighborliness}, where the input parameter for sparsity $\widehat{k}$ is overshooted. See also Section \ref{sec:practice} for more information.
We compare the above algorithms w.r.t. function values decrease and running times.

Figure \ref{fig:linreg1} depicts the summary of results we observed for the case $n = 2\cdot 10^5$, $m = 7500$ and $k = 500 $.
For IHT and accelerated IHT, we also consider the case where the input parameter for sparsity is set to $\widehat{k} = 2441 > k$. 
The graphs indicate that the accelerated hard thresholding technique can be much more efficient and scalable than the rest of the algorithms, while at the same time being at least as good in support/``signal" recovery performance.
For instance, while Algorithm \ref{algo:alps} is only $1.2 \times $ faster than IHT, when $k$ is known exactly, Algorithm \ref{algo:alps} is more resilient at overshooting $k$: in that case, IHT could take $> 2 \times$ time to get to the same level of accuracy.
At the same time, Algorithm \ref{algo:alps} detects much faster the correct support, compared to plain IHT.
Compared to FW methods (right plot), Algorithm \ref{algo:alps} is at least $10 \times$ faster than FW variants.

As stated before, we only focus on the optimization efficiency of the algorithms, not their statistical efficiency. 
That being said, we consider settings that are above the phase retrieval curve \cite{}, and here we make no comparisons and claims regarding the number of samples required to complete the sparse linear regression task. 
We leave such work for an extended version of this work.

%\textcolor{magenta}{Sparse constraints and regularized logistic regression with $\ell_2$-norm. This is not doable and easy to set up in the convex setting with $\ell_1$ and $\ell_2$ norms at the same time.}

\emph{Sparse linear regression with correlated regressors:} 
In this section, we test Algorithm~\ref{algo:alps} for support recovery, generalization and training loss performance in the sparse linear regression, under a different data generation setting. 
We generate the data as follows. 
We generate the feature matrix $800 \times  200$ design matrix $\Phi$ according to a first order auto-regressive process with $\texttt{correlation} = 0.4$. 
This ensures features are correlated with each other, which further makes feature selection a non-trivial task.
We normalize the feature matrix so that each feature has $\ell_2$-norm equal to one. 
We generate an arbitrary weight vector $x^\star$ with $\| x^\star\|_0 = 20$ and $\|x^\star\|_2=1$. 
The response vector $b$ is then computed as $y = \Phi x^\star + \varepsilon$, where $\varepsilon$ is guassian iid noise that is generated to ensure that the signal-to-noise ratio is $10$. 
Finally, the generated data is randomly split $50$-$50$ into training and test sets. 

We compare against Lasso \cite{scikit-learn}, oblivious greedy selection (Oblivious \cite{das2011submodular}), forward greedy selection (Greedy~\cite{das2011submodular}), and forward backward selection (FoBa~\cite{zhang2009adaptive}). The metrics we use to compare on are the generalization accuracy ($R^2$ coefficient determination performance on test set), recovery of true support (AUC metric on predicted support vs. true support), and training data fit (log likelihood on the training set). 
The results are presented in Figure~\ref{fig:linreg2}, and shows Algorithm \ref{algo:alps} performs very competitively: it is almost always better or equal to other methods across different sparsity levels.

%% file: conclusions.tex
%!TEX root = AccIHT.tex

\section{Overview and future directions}

The use of hard-thresholding operations is widely known.
In this work, we study acceleration techniques in simple hard-thresholding gradient procedures and characterize their performance; to the best of our knowledge, this is the first work to provide theoretical support for these type of algorithms.
Our preliminary results show evidence that machine learning problems can be efficiently solved using our accelerated non-convex variant, which is at least competitive with state of the art and comes with convergence guarantees. 

Our approach shows linear convergence; however, in theory, the acceleration achieved has dependence on the condition number of the problem not better than plain IHT.
This leaves open the question on what types of conditions are sufficient to guarantee the better acceleration of momentum in such non-convex settings? 

%\emph{Related work on dynamical systems and numerical analysis.}
Apart from the future directions ``scattered" in the main text, another possible direction lies at the intersection of dynamical systems and numerical analysis with optimization.
%Multi-step schemes originate from explicit finite differences discretization of dynamical systems; 
%\emph{e.g.}, the so-called Heavy-ball method \cite{polyak1964some} origins from the discretization of the friction dynamical system
%$\ddot{x}(t) + \gamma \dot{x}(t) + \nabla f(x(t)) = 0$, where $\gamma > 0$ plays the role of friction.
Recent developments on this subject can be found in \cite{wilson2016lyapunov} and \cite{scieur2017integration}. % use multi-step methods from numerical analysis to discretize the gradient flow equation.
We believe that extending these ideas in non-convex domains is interesting to better understand when and why momentum methods work in practical structured scenaria.

%% file: proofs2.tex
%!TEX root = AccIHT.tex

\section{Proof of Lemma \ref{thm:iter_inv}}{\label{sec:proofs}}

We start by observing that: $\|x_{i+1} - \bar{u}_i\|_2^2 \leq \|x^\star - \bar{u}_i \|_2^2$, due to the exactness of the hard thresholding operation.
Note that this holds for most $\mathcal{A}$ of interest\footnote{For example, in the case of matrices and low-rankness, this operation holds due to the Eckart-Young-Mirsky-Steward theorem, and the inner products of vectors naturally extend to the inner products over matrices.} but, again for simplicity, we will focus on the sparsity model here.
By adding and subtracting $x^\star$ and expanding the squares, we get:
\begin{align*}
\|x_{i+1} - x^\star + x^\star - \bar{u}_i\|_2^2 &\leq \|x^\star - \bar{u}_i \|_2^2 \Rightarrow \\
\|x_{i+1} - x^\star\|_2^2 + \|x^\star - \bar{u}_i\|_2^2 + 2 \left\langle x_{i+1} - x^\star, ~x^\star - \bar{u}_i \right \rangle &\leq \|x^\star - \bar{u}_i\|_2^2 \Rightarrow \\
\|x_{i+1} - x^\star\|_2^2 &\leq 2 \left \langle x_{i+1} - x^\star, ~\bar{u}_i - x^\star \right \rangle
\end{align*}

We also know that:
\begin{align*}
\bar{u}_i &= u_i - \mu_i \nabla_{\mathcal{T}_i} f(u_i) \\
 	      &= u_i + \mu_i \Phi_{\mathcal{T}_i}^\top(b - \Phi_{\mathcal{T}_i} u_i) \\
	      &= u_i + \mu_i \Phi_{\mathcal{T}_i}^\top(\Phi_{\mathcal{T}_i} x^\star + \varepsilon - \Phi_{\mathcal{T}_i} u_i) \\
	      &= u_i + \mu_i \Phi_{\mathcal{T}_i}^\top\Phi_{\mathcal{T}_i} \left(x^\star - u_i \right) + \mu_i \Phi_{\mathcal{T}_i}^\top \varepsilon
\end{align*}
Going back to the original inequality, we obtain:
\begin{align*}
\|x_{i+1} - x^\star\|_2^2 &\leq 2 \left \langle x_{i+1} - x^\star, ~u_i + \mu_i \Phi_{\mathcal{T}_i}^\top\Phi_{\mathcal{T}_i} \left(x^\star - u_i \right) + \mu_i \Phi_{\mathcal{T}_i}^\top \varepsilon - x^\star \right \rangle \\
								  &= 2 \left \langle x_{i+1} - x^\star, ~\left(I - \mu_i \Phi_{\mathcal{T}_i}^\top\Phi_{\mathcal{T}_i}\right)\left(u_i - x^\star\right)\right \rangle + 2\left \langle x_{i+1} - x^\star, ~\mu_i \Phi_{\mathcal{T}_i}^\top \varepsilon \right \rangle \\
								  &\leq 2 \left\| x_{i+1} - x^\star \right \|_2 \cdot \left\| \left(I - \mu_i \Phi_{\mathcal{T}_i}^\top\Phi_{\mathcal{T}_i}\right)\left(u_i - x^\star\right)\right\|_2 + 2\mu_i \left \| \Phi_{\mathcal{T}_i} \left(x_{i+1} - x^\star\right) \right\|_2 \cdot \left\|\varepsilon \right\|_2 \\
								  &\leq 2 \left\| x_{i+1} - x^\star \right\|_2 \cdot \left\| I - \mu_i \Phi_{\mathcal{T}_i}^\top\Phi_{\mathcal{T}_i} \right\|_2 \cdot \left\| u_i - x^\star \right\|_2 + 2\mu_i \left \| \Phi_{\mathcal{T}_i} \left(x_{i+1} - x^\star\right) \right\|_2 \cdot \left\|\varepsilon \right\|_2
\end{align*}
Using the properties of generalized restricted isometry property, we observe that: $\left \| \Phi_{\mathcal{T}_i} \left(x_{i+1} - x^\star\right) \right\|_2 \leq \sqrt{\beta_{2k}} \|x_{i+1} - x^\star\|_2$, and:
\begin{align*}
\left\| I - \mu_i \Phi_{\mathcal{T}_i}^\top\Phi_{\mathcal{T}_i} \right\|_2 \leq \frac{\beta_{3k} - \alpha_{3k}}{\alpha_{3k} + \beta_{3k}}
\end{align*}
In the last inequality we use the generalized RIP:
\begin{align*}
\alpha_{rk} \|x\|_2^2 \leq \|\Phi x\|_2^2 \leq \beta_{rk} \|x\|_2^2,
\end{align*}
and the fact that we can minimize with respect to $\mu_i$ to get:
\begin{align*}
\min_{\mu_i} \left\| I - \mu_i \Phi_{\mathcal{T}_i}^\top\Phi_{\mathcal{T}_i} \right\|_2 \leq \min_{\mu_i} \max\{\mu_i \beta_{3k} - 1, ~1 - \mu_i \alpha_{3k} \} = \frac{\beta_{3k} - \alpha_{3k}}{\alpha_{3k} + \beta_{3k}}
\end{align*}
for $\mu_i = \frac{2}{\alpha_{3k} + \beta_{3k}}$.

The above lead to:
\begin{align*}
\|x_{i+1} - x^\star\|_2^2 &\leq 2 \left\| x_{i+1} - x^\star \right\|_2 \tfrac{\beta_{3k} - \alpha_{3k}}{\alpha_{3k} + \beta_{3k}} \cdot \left\| u_i - x^\star \right\|_2 + \tfrac{2 \sqrt{\beta_{2k}}}{\alpha_{3k} + \beta_{3k}} \|x_{i+1} - x^\star\|_2 \cdot \left\|\varepsilon \right\|_2 \Rightarrow \\
\|x_{i+1} - x^\star\|_2 &\leq \tfrac{2(\beta_{3k} - \alpha_{3k})}{\alpha_{3k} + \beta_{3k}} \cdot \left\| u_i - x^\star \right\|_2 + \tfrac{2 \sqrt{\beta_{2k}}}{\alpha_{3k} + \beta_{3k}} \left\|\varepsilon \right\|_2
\end{align*}
Defining $\kappa = \frac{\beta_{3k}}{\alpha_{3k}}$ as the condition number, we get:
\begin{align*}
\|x_{i+1} - x^\star\|_2 &\leq \tfrac{2(\kappa - 1)}{\kappa + 1} \cdot \left\| u_i - x^\star \right\|_2 + \tfrac{2 \sqrt{\beta_{2k}}}{\alpha_{3k} + \beta_{3k}} \left\|\varepsilon \right\|_2
\end{align*}

Focusing on the norm term on RHS, we observe:
\begin{align*}
\|u_i - x^\star\|_2 &= \|x_i + \tau_i \left(x_i - x_{i-1}\right) - x^\star\|_2 = \|x_i + \tau_i \left(x_i - x_{i-1}\right) - (1 - \tau_i + \tau_i) x^\star\|_2\\
						&= \|(1 + \tau_i)(x_i - x^\star) + \tau_i(x^\star - x_{i-1})\|_2 \\
						&\leq |1 + \tau| \cdot \|x_i - x^\star\|_2 + |\tau| \cdot \|x_{i-1} - x^\star\|_2
\end{align*} 
where in the last inequality we used the triangle inequality and the fact that $\tau_i = \tau$, for all $i$.
Substituting this in our main inequality, we get:
\begin{align*}
\|x_{i+1} - x^\star\|_2 &\leq \tfrac{2(\kappa - 1)}{\kappa + 1} \cdot \left( |1 + \tau| \cdot \|x_i - x^\star\|_2 + |\tau| \cdot \|x_{i-1} - x^\star\|_2 \right) + \tfrac{2 \sqrt{\beta_{2k}}}{\alpha_{3k} + \beta_{3k}} \left\|\varepsilon \right\|_2\\
							  &= \tfrac{2(\kappa - 1)}{\kappa + 1} \cdot |1 + \tau| \cdot \|x_i - x^\star\|_2 + \tfrac{2(\kappa - 1)}{\kappa + 1} \cdot |\tau| \cdot \|x_{i-1} - x^\star\|_2 + \tfrac{2 \sqrt{\beta_{2k}}}{\alpha_{3k} + \beta_{3k}} \left\|\varepsilon \right\|_2
\end{align*}
Define $z(i) = \|x_i - x^\star\|_2$; this leads to the following second-order linear system:
\begin{align*}
z(i+1) \leq \tfrac{2(\kappa - 1)}{\kappa + 1} \cdot |1 + \tau| \cdot  z(i) + \tfrac{2(\kappa - 1)}{\kappa + 1} \cdot  |\tau| \cdot  z(i-1) + \tfrac{2 \sqrt{\beta_{2k}}}{\alpha_{3k} + \beta_{3k}} \left\|\varepsilon \right\|_2.
\end{align*}
We can convert this second-order linear system into a two-dimensional first-order system, where the variables of the linear system are multi-dimensional. 
To do this, we define a new state variable $w(i)$:
\begin{align*}
w(i) := z(i+1)
\end{align*} 
and thus $w(i+1) = z(i+2)$.
Using $w(i)$, we define the following 2-dimensional, first-order system:
\begin{align*}
\left\{
	\begin{array}{ll}
		w(i) - \tfrac{2(\kappa - 1)}{\kappa + 1} \cdot |1 + \tau| \cdot  w(i-1) - \tfrac{2(\kappa - 1)}{\kappa + 1} \cdot  |\tau| \cdot  z(i-1) - \tfrac{2 \sqrt{\beta_{2k}}}{\alpha_{3k} + \beta_{3k}} \left\|\varepsilon \right\|_2\leq 0, \\
		z(i) \leq w(i-1).
	\end{array}
\right.
\end{align*}
This further characterizes the evolution of two state variables, $\{ w(i), z(i) \}$:
\begin{align*}
\begin{bmatrix}
w(i) \\ z(i)
\end{bmatrix} &\leq
\begin{bmatrix}
\tfrac{2(\kappa - 1)}{\kappa + 1} \cdot |1 + \tau| & \tfrac{2(\kappa - 1)}{\kappa + 1} \cdot  |\tau| \\
1 & 0
\end{bmatrix} \cdot 
\begin{bmatrix}
w(i-1) \\
z(i-1)
\end{bmatrix} + 
\begin{bmatrix}
1 \\ 0
\end{bmatrix} \tfrac{2 \sqrt{\beta_{2k}}}{\alpha_{3k} + \beta_{3k}} \left\|\varepsilon \right\|_2 \Rightarrow \\
\begin{bmatrix}
\|x_{i+1} - x^\star\|_2 \\ \|x_i - x^\star\|_2
\end{bmatrix} &\leq \begin{bmatrix}
\tfrac{2(\kappa - 1)}{\kappa + 1} \cdot |1 + \tau| & \tfrac{2(\kappa - 1)}{\kappa + 1} \cdot  |\tau| \\
1 & 0
\end{bmatrix} \cdot 
\begin{bmatrix}
\|x_i - x^\star\|_2 \\
\|x_{i-1} - x^\star\|_2
\end{bmatrix} + 
\begin{bmatrix}
1 \\ 0
\end{bmatrix} \tfrac{2 \sqrt{\beta_{2k}}}{\alpha_{3k} + \beta_{3k}} \left\|\varepsilon \right\|_2,
\end{align*}
where in the last inequality we use the definitions $z(i) = \|x_i - x^\star\|_2$ and $w(i) = z(i+1)$.
Observe that the contraction matrix has non-negative values.
%By defining $y(i) = \begin{bmatrix} \|x_{i+1} - x^\star\|_2 \\ \|x_i - x^\star\|_2 \end{bmatrix} $ and after unfolding the recursion over time, we obtain:
%\begin{align*}
%y(i+1) \leq \underbrace{\begin{bmatrix}
%\left(1 - \tfrac{\alpha}{\beta}\right) \cdot (1 + \tau) & \left(1 - \tfrac{\alpha}{\beta}\right) \cdot  \tau \\
%1 & 0
%\end{bmatrix}^i}_{:=A^i} \cdot y(0)
%\end{align*} 
This completes the proof.

%% file: practice.tex
%!TEX root = AccIHT.tex

\section{Implementation details}{\label{sec:practice}}

So far, we have showed the theoretical performance of our algorithm, where several hyper-parameters are assumed known.
Here, we discuss a series of practical matters that arise in the implementation of our algorithm.

\subsection{Setting structure hyper-parameter $k$}

Given structure $\mathcal{A}$, one needs to set the ``succinctness" level $k$, as input to Algorithm \ref{algo:alps}.
Before we describe practical solutions on how to set up this value, we first note that selecting $k$ is often intuitively easier than setting the regularization parameter in convexified versions of \eqref{eq:problem}.
For instance, in vector sparsity settings for linear systems, where the Lasso criterion is used: $\argmin_{x \in \R^{n}} \left\{ \tfrac{1}{2} \|b - \Phi x\|_2^2 + \lambda \cdot \|x\|_1 \right\}$, selecting $\lambda > 0$ is a non-trivial task: 
arbitrary values of $\lambda$ lead to different $k$, and it is not obvious what is the ``sweet range" of $\lambda$ values for a particular sparsity level.
From that perspective, using $k$ leads to more interpretable results.

However, even in the strict discrete case, selecting $k$ could be considered art for many cases.
Here, we propose two ways for such selection: $(i)$ via cross-validation, and $(ii)$ by using phase transition plots. 
Regarding cross-validation, this is a well-known technique and we will skip the details; 
we note that we used cross-validation, as in \cite{jain2016structured}, to select the group sparsity parameters for the tumor classification problem in Subsection \ref{subsec:tumor}.

A different way to select $k$ comes from the recovery limits of the optimization scheme at hand: 
For simplicity, consider the least-squares objective with sparsity constraints. 
\cite{donoho2005neighborliness} describes mathematically the phase transition behavior of the basis pursuit algorithm \cite{chen2001atomic} for that problem; see Figure 1 in \cite{donoho2005neighborliness}. 
Moving along the phase transition line, triplets of $(m, n, k)$ can be extracted; for fixed $m$ and $n$, this results into a unique value for $k$. 
We used this ``overshooted" value in Algorithm \ref{algo:alps} at our experiments for sparse linear regression; see Figure \ref{fig:linreg1} in Subsection \ref{subsec:linreg}.
Our results show that, even using this procedure as a heuristic, it results into an automatic way of setting $k$, that leads to the correct solution.
Similar phase transition plots can be extracted, even experimentally, for other structures $\mathcal{A}$; see \emph{e.g.} \cite{recht2010guaranteed} for the case of low rankness.

\subsection{Selecting $\tau$ and step size in practice}

The analysis in Section \ref{sec:theory} suggests using $\tau$ within the range:
$|\tau| \leq \tfrac{(1 - \varphi \cdot \xi^{1/2})}{\varphi \cdot \xi^{1/2}}$. 
In order to compute the end points of this range, we require a good approximation of $\xi := 1 - \sfrac{\alpha}{\beta}$, where $\alpha$ and $\beta$ are the restricted strong convexity and smoothness parameters of $f$, respectively.

In general, computing $\alpha$ and $\beta$ in practice is an NP-hard task\footnote{To see this, in the sparse linear regression setting, there is a connection between $\alpha, \beta$ and the so-called restricted isometry constants \cite{tillmann2014computational}. It is well known that the latter is NP-hard to compute.}
A practical rule is to use a constant momentum term, like $\tau = \sfrac{1}{4}$: we observed that this value worked well in our experiments.\footnote{We did not perform binary search for this selection---we conjecture that better $\tau$ values in our results could result into even more significant gains w.r.t. convergence rates.}
%In the case where the $\tau$ selection 

In some cases, one can approximate $\alpha$ and $\beta$ with the smallest and largest eigenvalue of the hessian $\nabla^2 f(\cdot)$; \emph{e.g.}, in the linear regression setting, the hessian matrix is constant across all points, since $\nabla^2 f(\cdot) = \Phi^\top \Phi$.
This is the strategy followed in Subsection \ref{subsec:linreg} to approximate $\beta$ with $\widehat{\beta} := \lambda_{\max}(\Phi^\top \Phi)$. 
We also used $\sfrac{1}{\widehat{\beta}}$ as the step size.
Moreover, for such simplified but frequent cases, one can efficiently select step size and momentum parameter in closed form, via line search; see \cite{kyrillidis2011recipes}. 

In the cases where $\tau$ results into ``ripples" in the function values, we conjecture that the adaptive strategies in \cite{o2015adaptive} can be used to accelerate convergence. This solution is left for future work.

Apart from these strategies, common solutions for approximate $\alpha$ and $\beta$ include backtracking (update approximates of $\alpha$ and $\beta$ with per-iteration estimates, when local behavior demands it) \cite{becker2011templates, beck2009fast}, Armijo-style search tests, or customized tests (like eq. (5.7) in \cite{becker2011templates}).
However, since estimating the $\alpha$ parameter is a much harder task \cite{o2015adaptive, becker2011templates, nesterov2013gradient}, one can set $\tau$ as constant and focus on approximating $\beta$ for the step size selection.

%if an $\alpha$ term is not easy to be satisfied, use adaptive strategies from \cite{o2015adaptive} (see Algorithm 1), where tau is adaptively set up and is also adaptively restarted. 

%\emph{A practical rule for $\tau$:} Since the above expression is not entirely practical, here we discuss a practical way to set up $\tau$ (the same rule is used in the experiments): assuming bad conditioned problems (\emph{i.e.}, $\xi \approx 1$), we observe that the second rule in the above range translates into $\tau \lesssim -3 - 2\sqrt{2}$ and the third rule above translates into $\tau \gtrsim 0$. Thus, setting $\tau \in \left(0, \tfrac{|1 - \varphi \cdot \xi^{1/2}|}{\varphi \cdot \xi^{1/2}}\right)$ is a simplified and practical rule.

%\subsection{Finding good $\alpha$ and $\beta$ parameters}
%
%\textcolor{magenta}{See TFOCS, FISTA}
%\cite{nesterov2013gradient}

\subsection{Inexact projections $\Pi_{k, \mathcal{A}}(\cdot)$}

Part of our theory relies on the fact that the projection operator $\Pi_{k, \mathcal{A}}(\cdot)$ is exact.
We conjecture that our theory can be extended to \emph{approximate} projection operators, along the lines of 
\cite{blumensath2011sampling, shah2011iterative, kyrillidis2012combinatorial, hegde2015approximation}. We present some experiments that perform approximate projections for overlapping group sparse structures and show AccIHT can perform well. 
We leave the theoretical analysis as potential future work.

%% file: conv_appendix.tex
%!TEX root = AccIHT.tex

%\section{Convergence properties using the $i$-th power of contraction matrix}{\label{sec:conv_appendix}}
\section{Proof of Lemma \ref{lemma:nth-power2}}{\label{subsec:n-thpowerx}}

First, we state the following simple theorem; the proof is omitted.
\begin{lemma}{\label{lemma:nth-power1}}
Let $A := \begin{bmatrix}
\gamma & \delta \\
\epsilon & \zeta
\end{bmatrix}$ be a $2 \times 2$ matrix with distinct eigevalues $\lambda_1, ~\lambda_2$. 
Then, $A$ has eigenvalues such that: 
\begin{align*}
\lambda_{1,2} = \tfrac{\omega}{2} \pm \sqrt{\tfrac{\omega^2}{4} - \Delta},
\end{align*}
where $\omega := \gamma + \zeta$ and $\Delta = \gamma \cdot \zeta - \delta \cdot \epsilon$.
\end{lemma}

We will consider two cases: $(i)$ when $\lambda_1 \neq \lambda_2$ and, $(ii)$ when $\lambda_1 = \lambda_2$.
\subsubsection{$\lambda_1 \neq \lambda_2$}
Some properties regarding these two eigenvalues are the following:
\begin{align*}
\lambda_1 + \lambda_2 = \left(\tfrac{\omega}{2} + \sqrt{\tfrac{\omega^2}{4} - \Delta}\right) + \left(\tfrac{\omega}{2} - \sqrt{\tfrac{\omega^2}{4} - \Delta}\right) = \omega
\end{align*} 
and
\begin{align*}
\lambda_1 \lambda_2 = \left(\tfrac{\omega}{2} + \sqrt{\tfrac{\omega^2}{4} - \Delta}\right) \cdot \left(\tfrac{\omega}{2} - \sqrt{\tfrac{\omega^2}{4} - \Delta}\right) = \tfrac{\omega^2}{4} - \tfrac{\omega^2}{4} + \Delta = \Delta
\end{align*}

Let us define:
\begin{align*}
B_1 &= -(A - \lambda_1 \cdot I) \\
B_2 &= (A - \lambda_2 \cdot I)
\end{align*}
Observe that:
\begin{align*}
\lambda_2 \cdot B_1 + \lambda_1 \cdot B_2 &= - \lambda_2 \left(A - \lambda_1 \cdot I \right) + \lambda_1 \left(A - \lambda_2 \cdot I\right) \\
																&= -\lambda_2 A + \lambda_1 \lambda_2 I + \lambda_1A - \lambda_1 \lambda_2 I \\ 
																&= (\lambda_1 - \lambda_2)A
\end{align*}
which, under the assumption that $\lambda_1 \neq \lambda_2$, leads to:
\begin{align*}
A = \tfrac{\lambda_2}{\lambda_1 - \lambda_2} B_1 + \tfrac{\lambda_1}{\lambda_1 - \lambda_2} B_2
\end{align*}

Furthermore, we observe:
\begin{align*}
B_1 \cdot B_1 &= (A - \lambda_1 \cdot I) \cdot (A - \lambda_1 \cdot I) \\
				      &= A^2 + \lambda_1^2 \cdot I - 2\lambda_1 A
\end{align*}
By the Calley-Hamilton Theorem on $2 \times 2$ matrices, we know that the characteristic polynomial $p(A) = A^2 - \text{Tr}(A) \cdot A - \texttt{det}(A) \cdot I = 0$, and thus,
\begin{align*}
A^2 &= (\gamma + \zeta) \cdot A - \left(\gamma \cdot \zeta - \delta \cdot \epsilon \right) \cdot I \Rightarrow \\
       &= \omega \cdot A - \Delta \cdot I
\end{align*}
Using the $\omega$ and $\Delta$ characterizations above w.r.t. the eigenvalues $\lambda_{1,2}$, we have:
\begin{align*}
A^2 = (\lambda_1 + \lambda_2) \cdot A - \lambda_1 \lambda_2 I
\end{align*} 
and thus:
\begin{align*}
B_1 \cdot B_1 &= (\lambda_1 + \lambda_2) A - \lambda_1 \lambda_2 I + \lambda_1^2 I - 2\lambda_1 A \\
					 &= (\lambda_2 - \lambda_1) A - (\lambda_1 \lambda_2 - \lambda_1^2)\cdot I \\
					 &= (\lambda_2 - \lambda_1) \cdot (A - \lambda_1 I) \\
					 &= (\lambda_1 - \lambda_2) \cdot B_1
\end{align*}
Similarly, we observe that:
\begin{align*}
B_2 \cdot B_2 = \cdots = (\lambda_1 - \lambda_2) \cdot B_2
\end{align*}
On the other hand, the cross product $B_1 \cdot B_2 = 0$.
To see this:
\begin{align*}
B_1 \cdot B_2 &= -(A - \lambda_1 I) \cdot (A - \lambda_2 I) \\
					 &= -A^2 - \lambda_1 \lambda_2 I + \lambda_2 A + \lambda_1 A \\
					 &= -A^2 + (\lambda_1 + \lambda_2) A - \lambda_1 \lambda_2 I = 0
\end{align*}
by the Calley-Hamilton Theorem.
Given the above, we have:
\begin{align*}
B_1^2 &= B_1 \cdot B_1 = (\lambda_1 - \lambda_2) \cdot B_1 \\
\vspace{0.2cm}
B_1^3 &= B_1^2 \cdot B_1 = (\lambda_1 - \lambda_2) \cdot B_1 = (\lambda_1 - \lambda_2)^2 \cdot B_1 \\
\vdots \\
B_1^i &= \cdots = (\lambda_1 - \lambda_2)^{i-1} B_1
\end{align*}
Similarly for $B_2$:
\begin{align*}
B_2^i = (\lambda_1 - \lambda_2)^{i-1} B_2
\end{align*} 
Getting back to the characterization of $A$ via $B_1$ and $B_2$, $A = \tfrac{\lambda_2}{\lambda_1 - \lambda_2} B_1 + \tfrac{\lambda_1}{\lambda_1 - \lambda_2} B_2$, and given that any cross product of $B_1 \cdot B_2 = 0$, it is easy to see that $A^i$ equals to:
\begin{align*}
A &= \left(\tfrac{\lambda_2}{\lambda_1 - \lambda_2}\right)^i \cdot B_1^i + \left(\tfrac{\lambda_1}{\lambda_1 - \lambda_2}\right)^i \cdot B_2^i \\
   &= \left(\tfrac{\lambda_2}{\lambda_1 - \lambda_2}\right)^i \cdot (\lambda_1 - \lambda_2)^{i-1} B_1 + \left(\tfrac{\lambda_1}{\lambda_1 - \lambda_2}\right)^i \cdot (\lambda_1 - \lambda_2)^{i-1} B_2 \\
   &= \tfrac{\lambda_2^i}{\lambda_1 - \lambda_2} B_1 + \tfrac{\lambda_1^i}{\lambda_1 - \lambda_2} B_2 \\
   &= \tfrac{\lambda_2^i}{\lambda_1 - \lambda_2} \cdot (-A + \lambda_1 I) + \tfrac{\lambda_1^i}{\lambda_1 - \lambda_2} \cdot (A - \lambda_2 I) \\
   &= \tfrac{\lambda_1^i - \lambda_2^i}{\lambda_1 - \lambda_2} \cdot A + \left(\lambda_1 \cdot \tfrac{\lambda_2^i}{\lambda_1 \lambda_2} - \lambda_2 \cdot \tfrac{\lambda_1^i}{\lambda_1 - \lambda_2} \right) \cdot I \\
   &= \tfrac{\lambda_1^i - \lambda_2^i}{\lambda_1 - \lambda_2} \cdot A - \lambda_1 \lambda_2 \cdot \tfrac{\lambda_1^{i-1} - \lambda_2^{i-1}}{\lambda_1 - \lambda_2} \cdot I
\end{align*}
where in the fourth equality we use the definitions of $B_1$ and $B_2$.

\subsubsection{$\lambda_1 = \lambda_2$}
In this case, let us denote for simplicity: $\lambda \equiv \lambda_1 = \lambda_2$. 
By the Calley-Hamilton Theorem, we have:
\begin{align*}
A^2 = 2\lambda \cdot A - \lambda^2 \cdot I \Longrightarrow (A - \lambda \cdot I)^2 = 0
\end{align*}
Let us denote $C = A - \lambda \cdot I$. 
From the derivation above, it holds:
\begin{align*}
C^2 &= (A - \lambda \cdot I)^2 = 0 \\
C^3 &= C^2 \cdot C = 0 \\
&\vdots \\
C^i &= C^{i-1} \cdot C = 0.
\end{align*} 
Thus, as long as $i \geq 2$, $C^i = 0$.
Focusing on the $i$-th power of $A$, we get:
\begin{align*}
A^i = \left(A + \lambda \cdot  - \lambda \cdot I \right)^i = \left( C + \lambda \cdot I \right)^i 
\end{align*}
By the multinomial theorem, the above lead to:
\begin{align*}
A^i = \sum_{|\theta| = i} {i \choose \theta} \left(C \cdot (\lambda \cdot I)\right)^\theta,
\end{align*} 
where $\theta = (\theta_1, \theta_2)$ and $\left(C \cdot (\lambda \cdot I)\right)^\theta = C^{\theta_1} \cdot \left(\lambda \cdot I \right)^{\theta_2}$, according to multi-indexes notations.
However, we know that only when $i < 2$, $C^i$ could be nonzero. 
This translates into keeping only two terms in the summation above:
\begin{align*}
A^i = \lambda^i \cdot I + i \cdot \lambda^{i-1} \cdot C = \lambda^i \cdot I + i \cdot \lambda^{i-1} \cdot (A - \lambda \cdot I)
\end{align*} 

%% file: counter.tex
%!TEX root = AccIHT.tex

\section{Non-increasing function values and momentum}{\label{sec:counter}}

\begin{wrapfigure}{r}{0.4\textwidth} 
\centering
\includegraphics[width=0.4\textwidth]{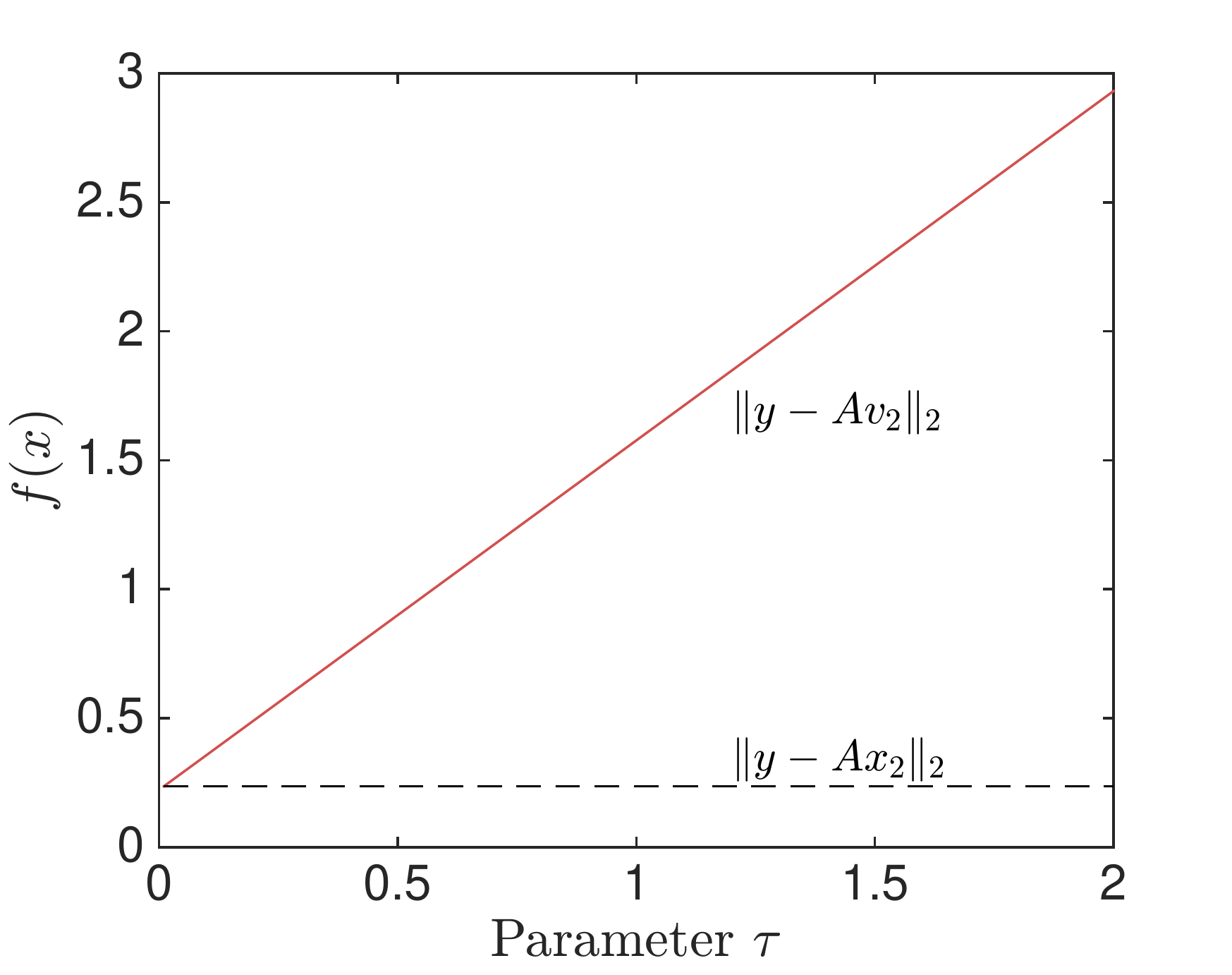} 
\caption{The use of momentum could be skipped in \cite{wei2015fast}.} 
\label{fig:counter}
\end{wrapfigure}
Here, we present a toy example for the analysis in \cite{wei2015fast}, where momentum term is not guaranteed to be used per step.
While this is not in general an issue, it might lead to repeatedly skipping the momentum term and, thus losing the acceleration.

Let us focus on the sparse linear regression problem, where the analysis in \cite{wei2015fast} applies. 
That means, $f(x) := \|b - \Phi x\|_2^2$, where $b \in \mathbb{R}^m$, $\Phi \in \mathbb{R}^{m \times n}$ and $x \in \mathbb{R}^n$. 
$b$ represents the set of observations, according to the linear model: $b = \Phi x^\star + \varepsilon$, where $x^\star$ is the sparse vector we look for and $\varepsilon$ is an additive noise term.
We assume that $m < n$ and, thus, regularization is needed in order to hope for a meaningful solution.

Similar to the algorithm considered in this paper, \cite{wei2015fast} performs a momentum step, where 
$v_{i+1} = x_{i+1} + \tau_{i+1} \cdot (x_{i+1} - x_i),$
where
\begin{align*}
\tau_{i+1} &= \argmin_{\tau} \| b - \Phi v_{i+1}\|_2^2 \\ &= \argmin_{\tau} \|b - \Phi \left(x_{i+1} + \tau \cdot (x_{i+1} - x_i)\right)\|_2^2
\end{align*}
The above minimization problem has a closed form solution.
However, the analysis in \cite{wei2015fast} assumes that $\|y - \Phi v_{i+1}\|_2 \leq \|y - \Phi x_{i+1}\|_2$, \emph{i.e.}, per iteration the momentum step does not increase the function value.

As we show in the toy example below, assuming positive momentum parameter $\tau \geq 0$, this assumption leads to no momentum term, when this is not satisfied. 
Consider the setting:
\begin{align*}
\underbrace{\begin{bmatrix}
    0.3870 \\
   -0.1514
\end{bmatrix}}_{=b} \approx 
\underbrace{\begin{bmatrix}
    0.3816  &  -0.2726  &  0.0077 \\
   -0.1598  &  1.9364  & -0.3908
\end{bmatrix}}_{=\Phi} \cdot 
\underbrace{\begin{bmatrix}
     1 \\
     0 \\
     0
\end{bmatrix}}_{=x^\star} + 
\underbrace{\begin{bmatrix}
0.0055 \\
    0.0084
\end{bmatrix}}_{=\varepsilon}
\end{align*} 
Further, assume that $x_1 = \begin{bmatrix} -1.7338 & 0 & 0 \end{bmatrix}^\top$ and $x_2 = \begin{bmatrix} 1.5415 & 0 & 0 \end{bmatrix}^\top$. 
Observe that $\|b - \Phi x_1\|_2 = 1.1328$ and $\|b - \Phi x_2\|_2 = 0.2224$, \emph{i.e.}, we are ``heading" towards the correct direction.
However, for any $\tau > 0$, $\|b - \Phi v_{2}\|_2$ increases; see Figure \ref{fig:counter}. 
This suggests that, unless there is an easy closed-form solution for $\tau$, setting $\tau$ differently does not guarantee that the function value $f(v_{i+1})$ will not increase, and the analysis in \cite{wei2015fast} does not apply.

%% file: moreexperiments.tex
%!TEX root = AccIHT.tex

\section{More experiments}\label{sec:moreexperiments}

\subsection{Group sparse, $\ell_2$-norm regularized logistic regression}{\label{subsec:tumor}}
For this task, we use the tumor classification on breast cancer dataset in \cite{jacob2009group} and test Algorithm \ref{algo:alps} on group sparsity model $\mathcal{A}$: we are interested in finding groups of genes that carry biological information such as regulation, involvement in the same chain of metabolic reactions, or protein-protein interaction.
We follow the procedure in \cite{jain2016structured}\footnote{\emph{I.e.}, 5-fold cross validation scheme to select parameters for group sparsity and $\ell_2$-norm regularization parameter - we use $\widehat{\beta}$ as in subsection~\ref{subsec:linreg}.} to extract misclassification rates and running times for FW variants, IHT and Algorithm \ref{algo:alps}. 
The groupings of genes are overlapping, which means that exact projections are hard to obtain. 
We apply the greedy projection algorithm of~\cite{jain2016structured} to obtain approximate projections. 
For cross-validating for the FW variants, we sweep over $\{10^{-3}, 10^{-2},10^{-1},1,10,100\}$ for regularization parameter, and $\{10^{-1},1,5,10, 50, 100\}$ for the scaling of the $\ell_1$ norm ball for sparsity inducing regularization. 
For IHT variants, we use the same set for the sweep for regularization parameter as we used for FW variants, and use $\{2, 5, 10, 15, 20, 50, 75, 100\}$ for sweep over the number of groups selected. 
After the best setting is selected for each algorithm, the time taken is calculated for time to convergence with the respective best parameters. 
The results are tabulated in Table~\ref{tab:tumor}. 
We note that this setup is out of the scope of our analysis, since our results assume exact projections. 
Nevertheless, we obtain competitive results suggesting that the acceleration scheme we propose for IHT warrants further study for the case of inexact projections. 

\begin{table}
\centering
\rowcolors{2}{white}{black!05!white}
\begin{tabular}{c c c c c}
	\toprule
	Algorithm & & Test error & & Time (sec) \\
	\cmidrule{1-1} \cmidrule{3-3} \cmidrule{5-5}
	FW \cite{lacoste2015global} & &  0.2938 & & 58.45 \\
	FW-Away \cite{lacoste2015global} & & 0.2938 & & 40.34 \\
	FW-Pair \cite{lacoste2015global} & & 0.2938 & &38.22 \\
    IHT \cite{jain2016structured} & & 0.2825 & & 5.24 \\
    \midrule
	Algorithm \ref{algo:alps} & & 0.2881 & & 3.45 \\
	\bottomrule
\end{tabular} \vspace{0.1cm}
\caption{Results for $\ell_2$-norm regularized logistic regression for tumor classification on the breast cancer dataset.}
\label{tab:tumor}
\end{table}

\subsection{Low rank image completion from subset of entries}

Here, we consider the case of matrix completion in low-rank, subsampled images. 
In particular, let $X^\star \in \R^{p \times n}$ be a low rank image; see Figures \ref{fig:image1_MC}-\ref{fig:image2_MC} for some ``compressed" low rank images in practice.
In the matrix completion setting, we observe only a subset of pixels in $X^\star$: $b = \mathcal{M}(X^\star)$ where $\mathcal{M}: \R^{p \times n} \rightarrow \R^m$ is the standard mask operation that down-samples $X^\star$ by selecting only $m \ll p \cdot n$ entries. 
The task is to recover $X^\star$ by minimizing $f(X) = \tfrac{1}{2} \|b - \mathcal{M}(X)\|_2^2$, under the low-rank model $\mathcal{A}$. 
According to \cite{negahban2012restricted}, such setting satisfies a slightly different restricted strong convexity/smoothness assumption; 
nevertheless, in Figures \ref{fig:MC_convrates}-\ref{fig:image2_MC} we demonstrate in practice that standard algorithms could still be applied: 
we compare accelerated IHT with plain IHT \cite{jain2010guaranteed}, an FW variant \cite{jaggi2013revisiting}, and the very recent matrix factorization techniques for low rank recovery (BFGD) \cite{bhojanapalli2016dropping, park2016finding}.
In our experiments, we use a line-search method for step size selection in accelerated IHT and IHT.
We observe the superior performance of accelerated IHT, compared to the rest of algorithms; it is notable to report that, for moderate problem sizes, non-factorized methods seem to have advantages in comparison to non-convex factorized methods, since low-rank projections (via SVD or other randomized algorithms) lead to significant savings in terms of number of iterations. 
Similar comparison results can be found in \cite{kyrillidis2014matrix, blanchard2015cgiht}.
Overall, it was obvious from our findings that Algorithm \ref{algo:alps} obtains the best performance among the methods considered.

\begin{figure*}[t!]
\centering
\includegraphics[width=0.48\textwidth, bb=0 0 500 500]{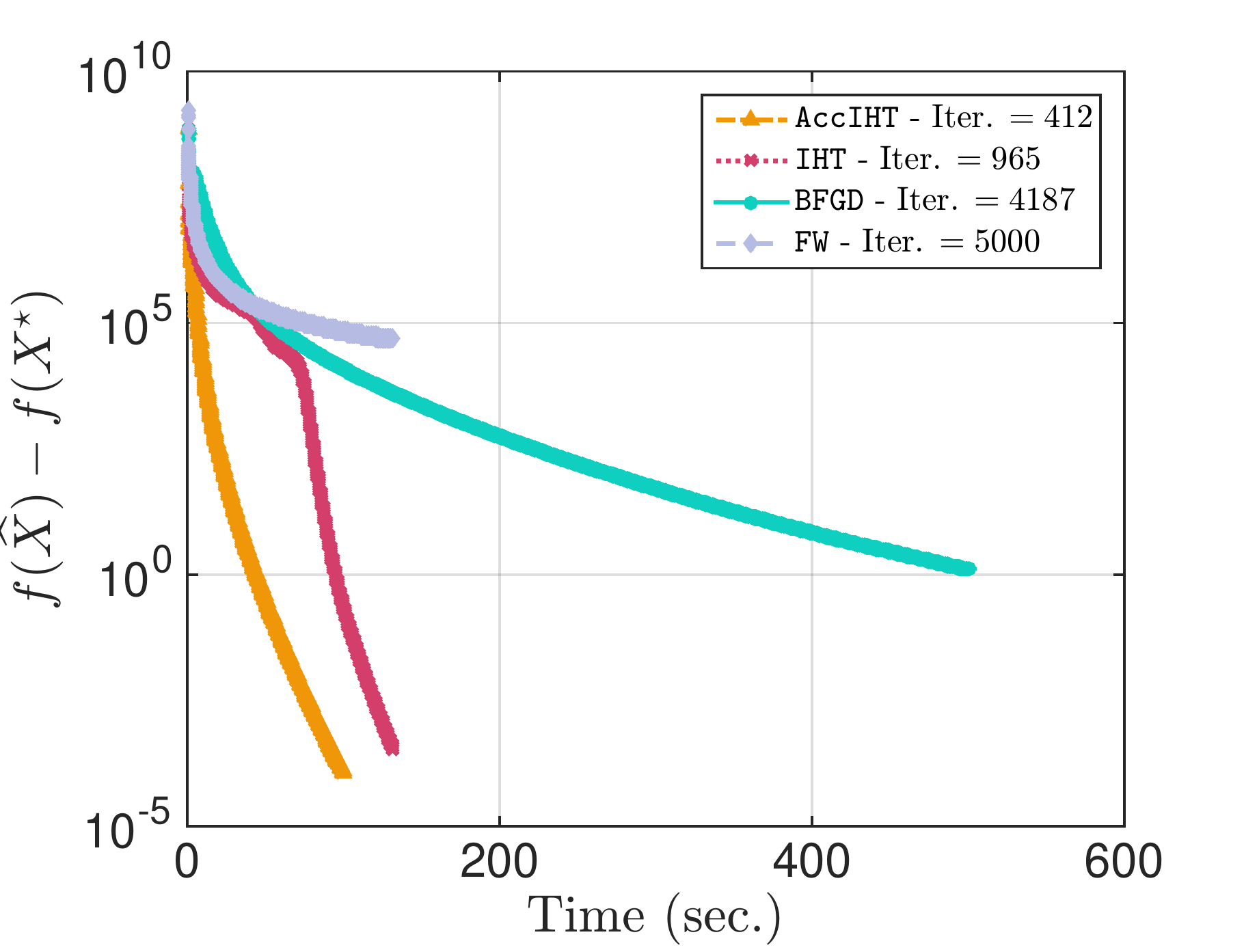} 
\includegraphics[width=0.48\textwidth, bb=0 0 500 500]{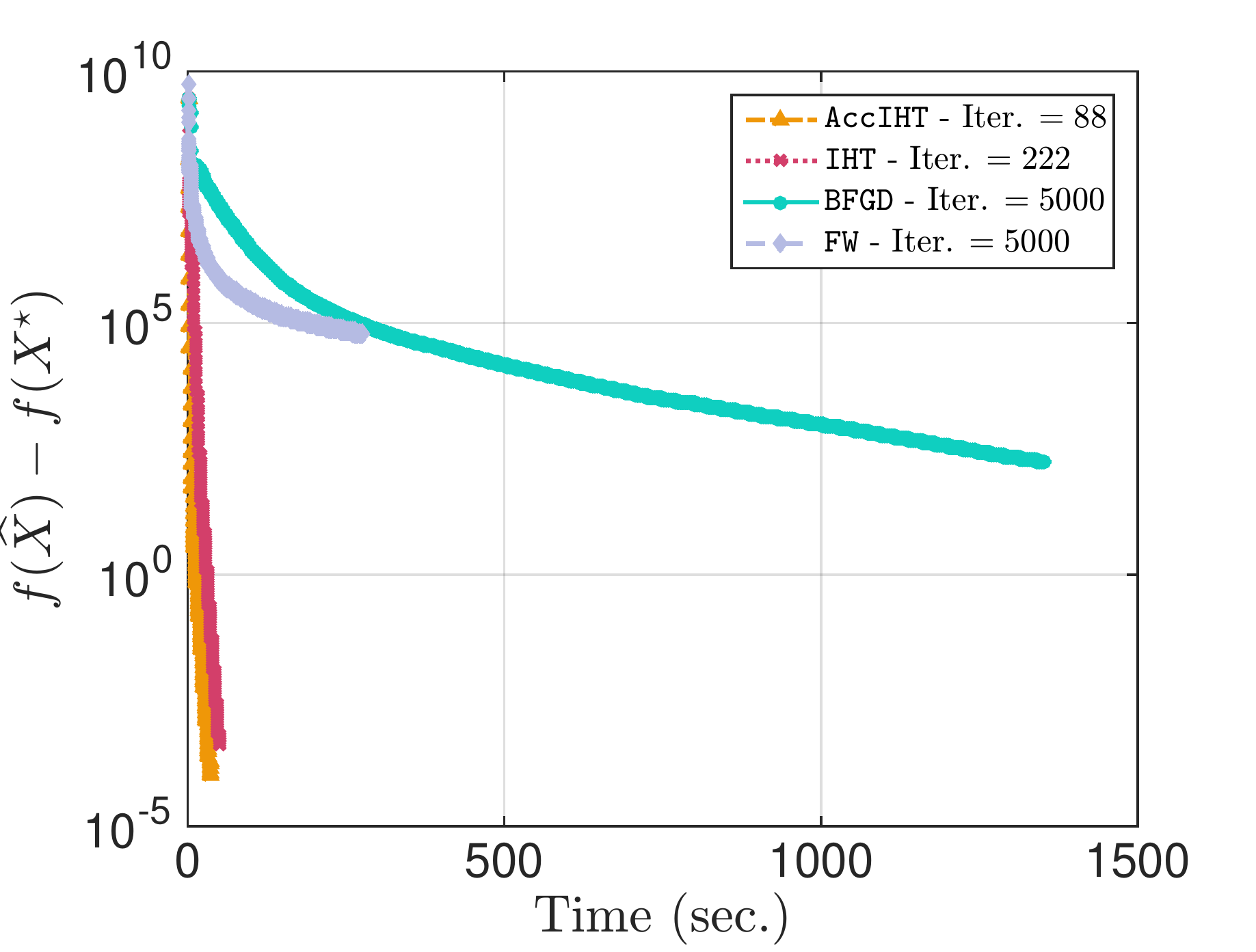} 
\caption{Time spent vs. function values gap $f(\widehat{x}) - f(x^\star)$. Left plot corresponds to the ``bikes" image, while the right plot to the ``children" image.}
\label{fig:MC_convrates}
\end{figure*}

\begin{figure*}[t!]
\centering
\includegraphics[width=0.30\textwidth, bb=0 0 700 700]{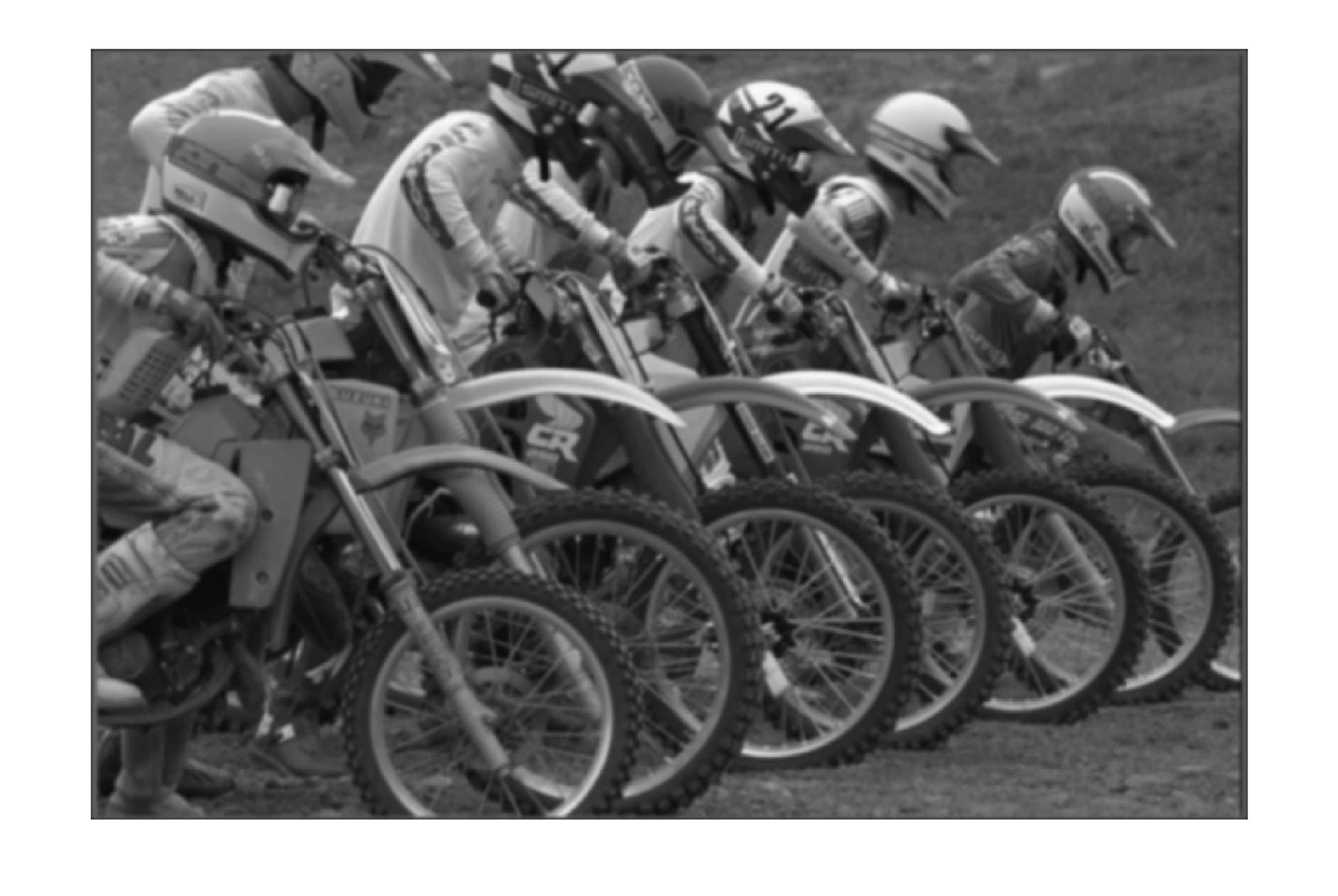} 
\includegraphics[width=0.30\textwidth, bb=0 0 700 700]{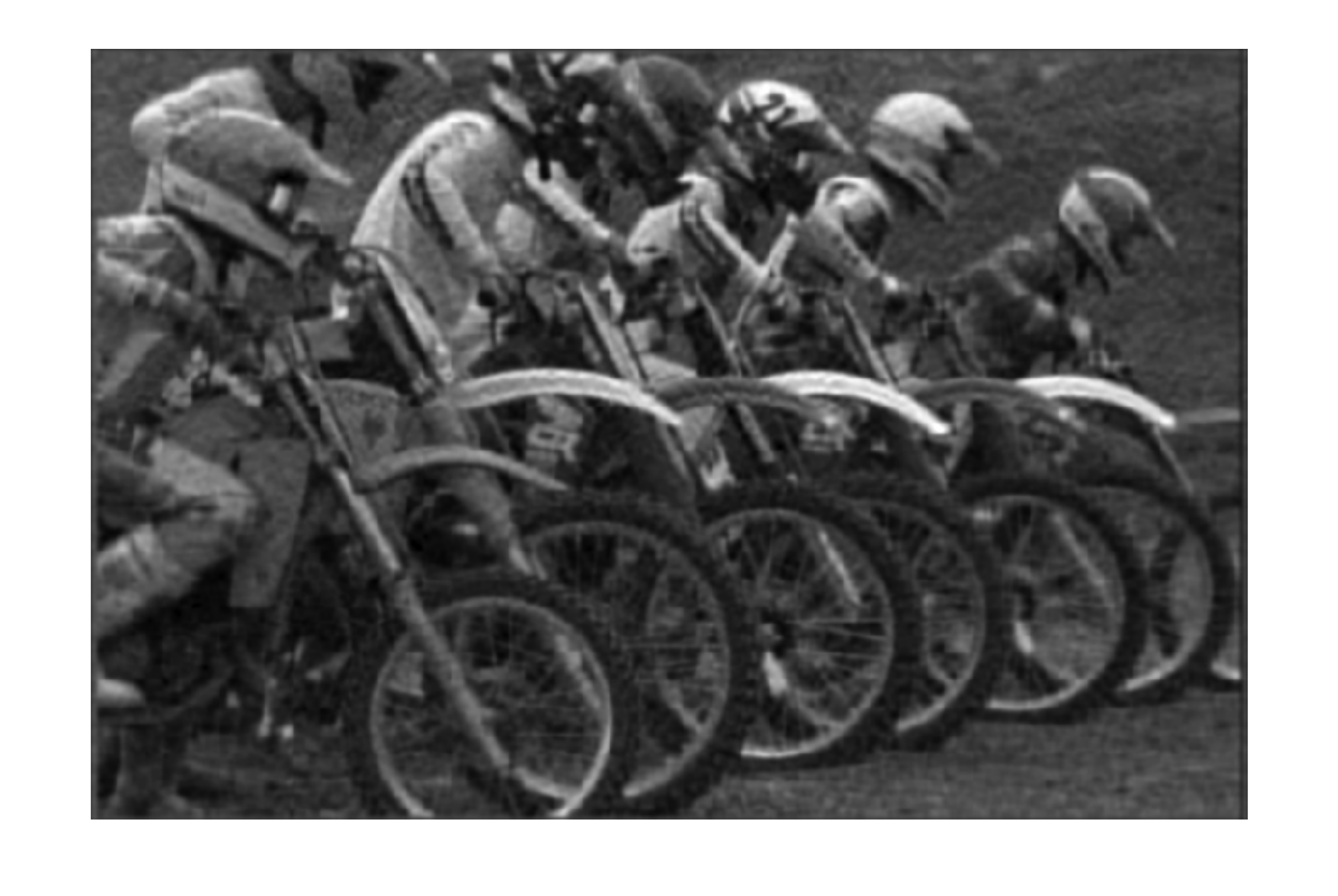} 
\includegraphics[width=0.30\textwidth, bb=0 0 700 700]{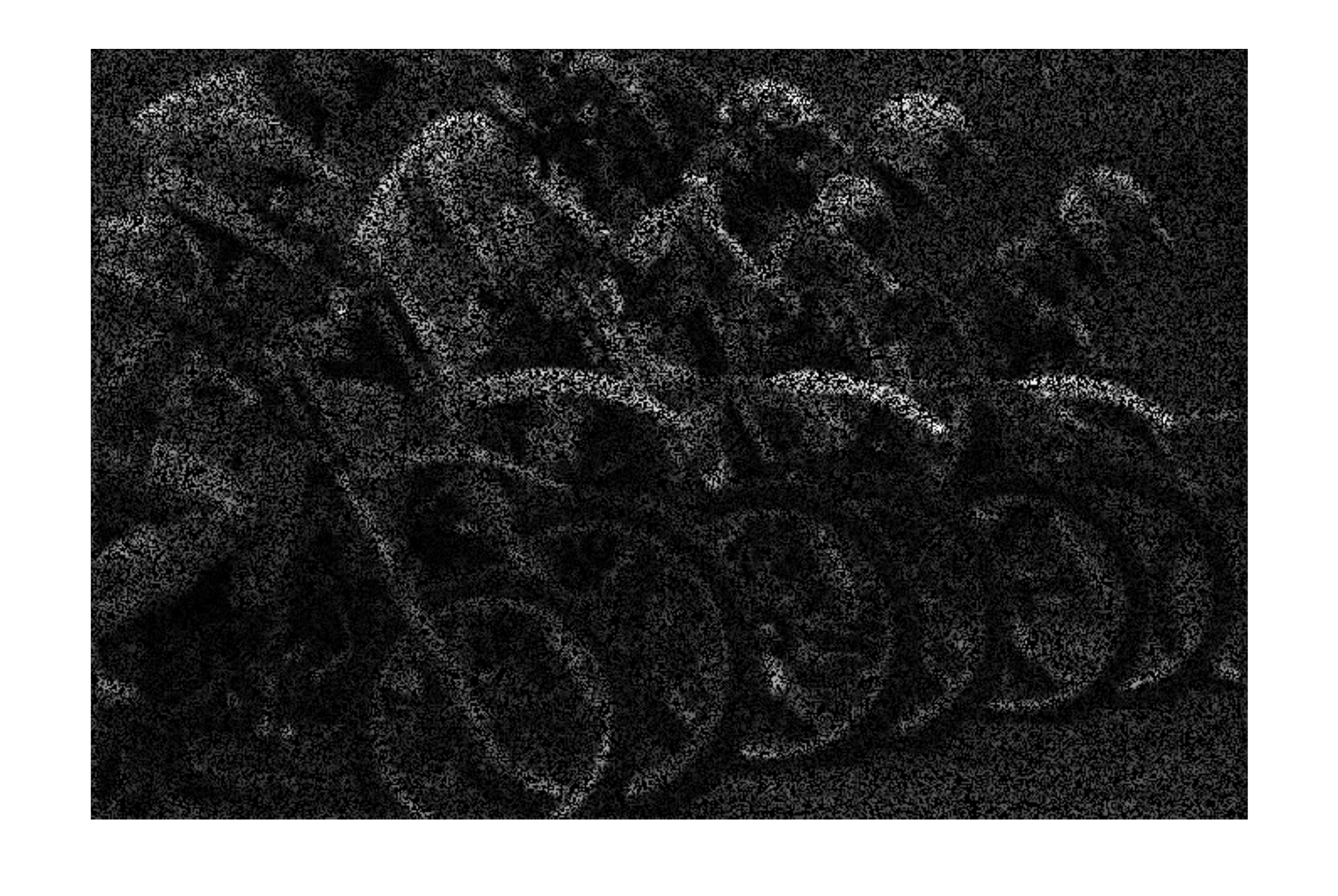} \\
\includegraphics[width=0.22\textwidth, bb=0 0 700 700]{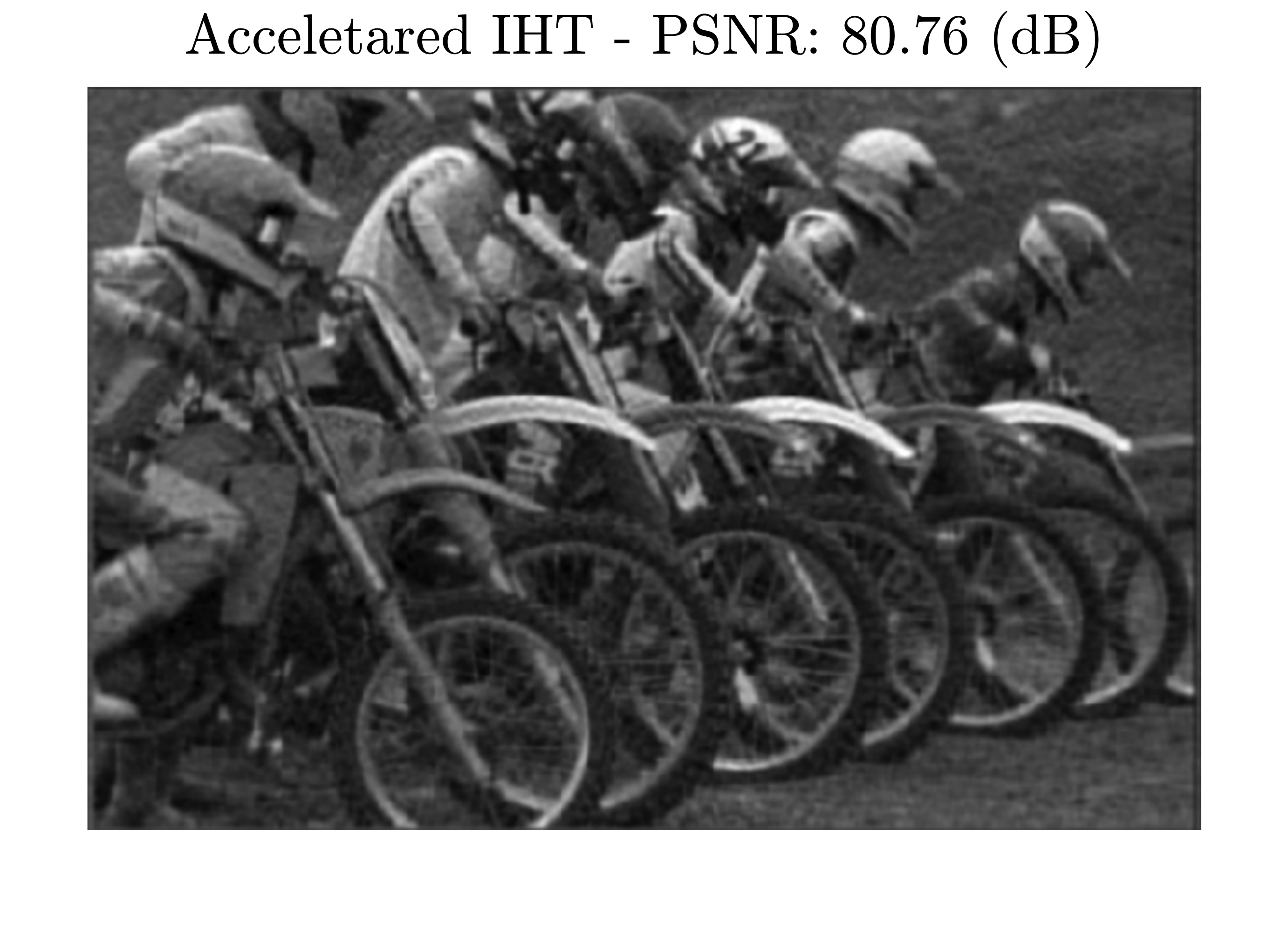} 
\includegraphics[width=0.22\textwidth, bb=0 0 700 700]{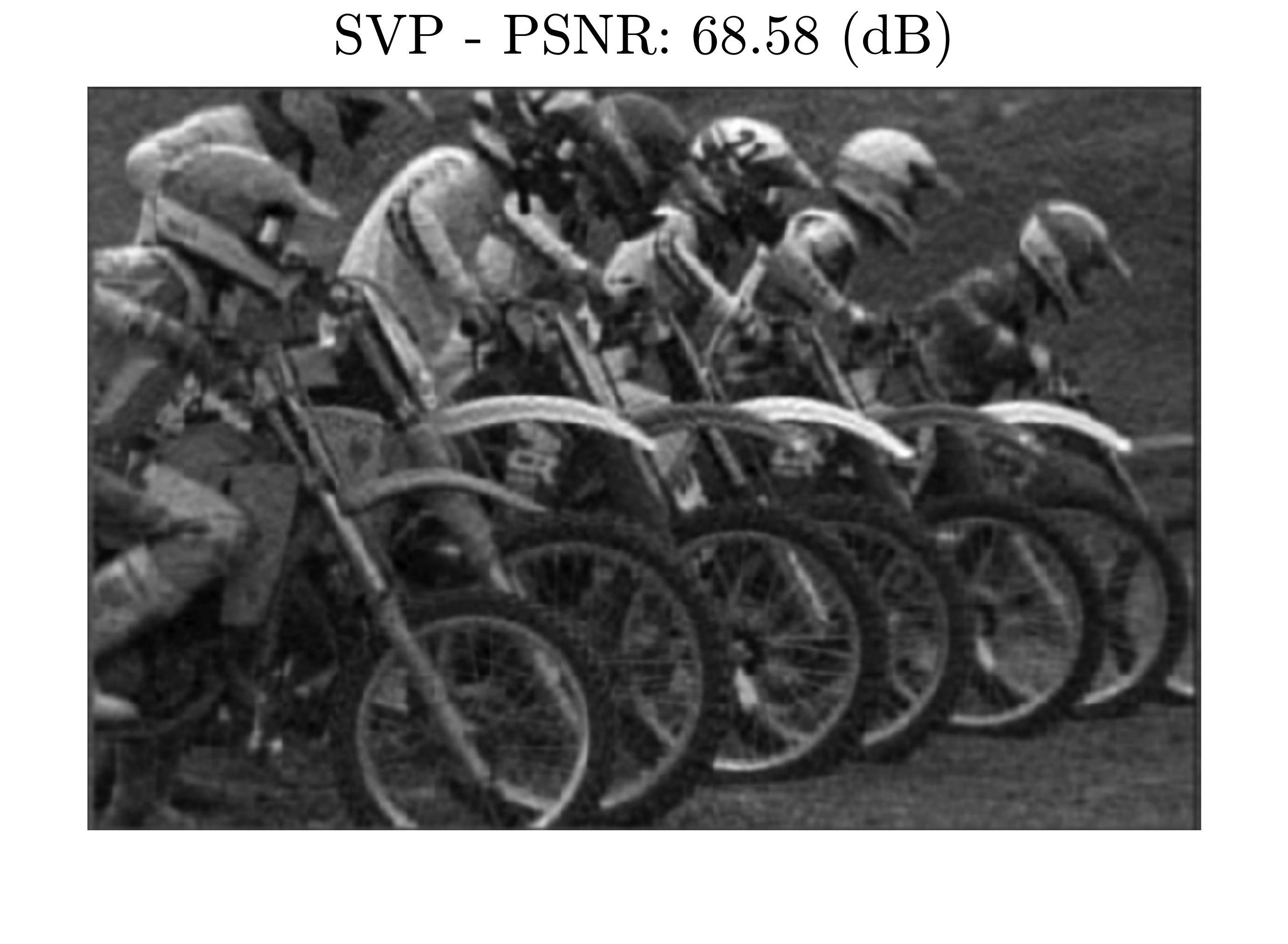} 
\includegraphics[width=0.22\textwidth, bb=0 0 700 700]{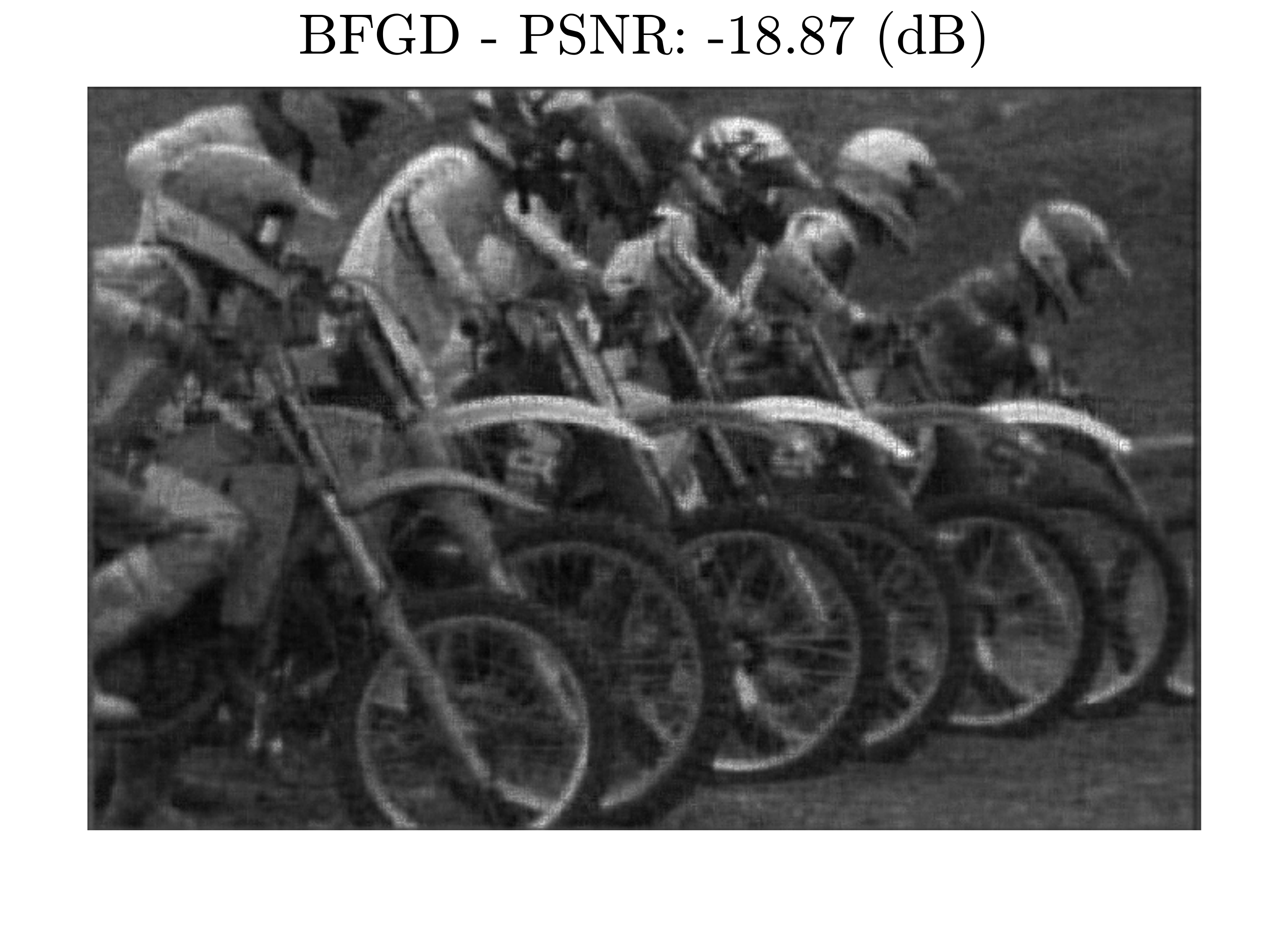} 
\includegraphics[width=0.22\textwidth, bb=0 0 700 700]{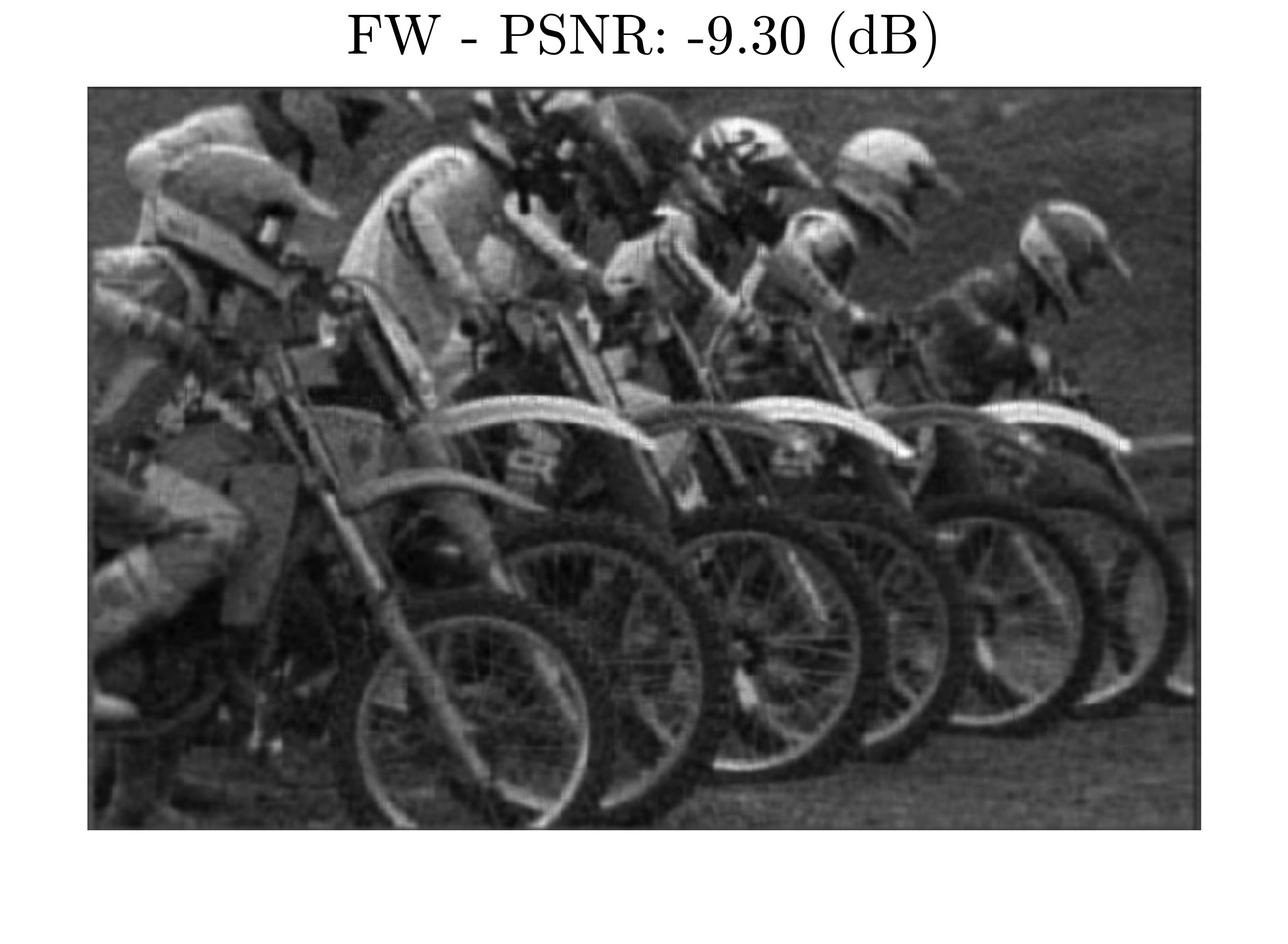}
\caption{Reconstruction performance in image denoising settings. The image size is $512 \times 768$ ($393,216$ pixels) and the approximation rank is preset to $r = 60$. We observe 35\% of the pixels of the true image. \textbf{Top row:} Original, low rank approximation, and observed image. \textbf{Bottom row:} Reconstructed images. }
\label{fig:image1_MC}
\end{figure*}

\begin{figure*}[b]
\centering
\includegraphics[width=0.30\textwidth, bb=0 0 700 700]{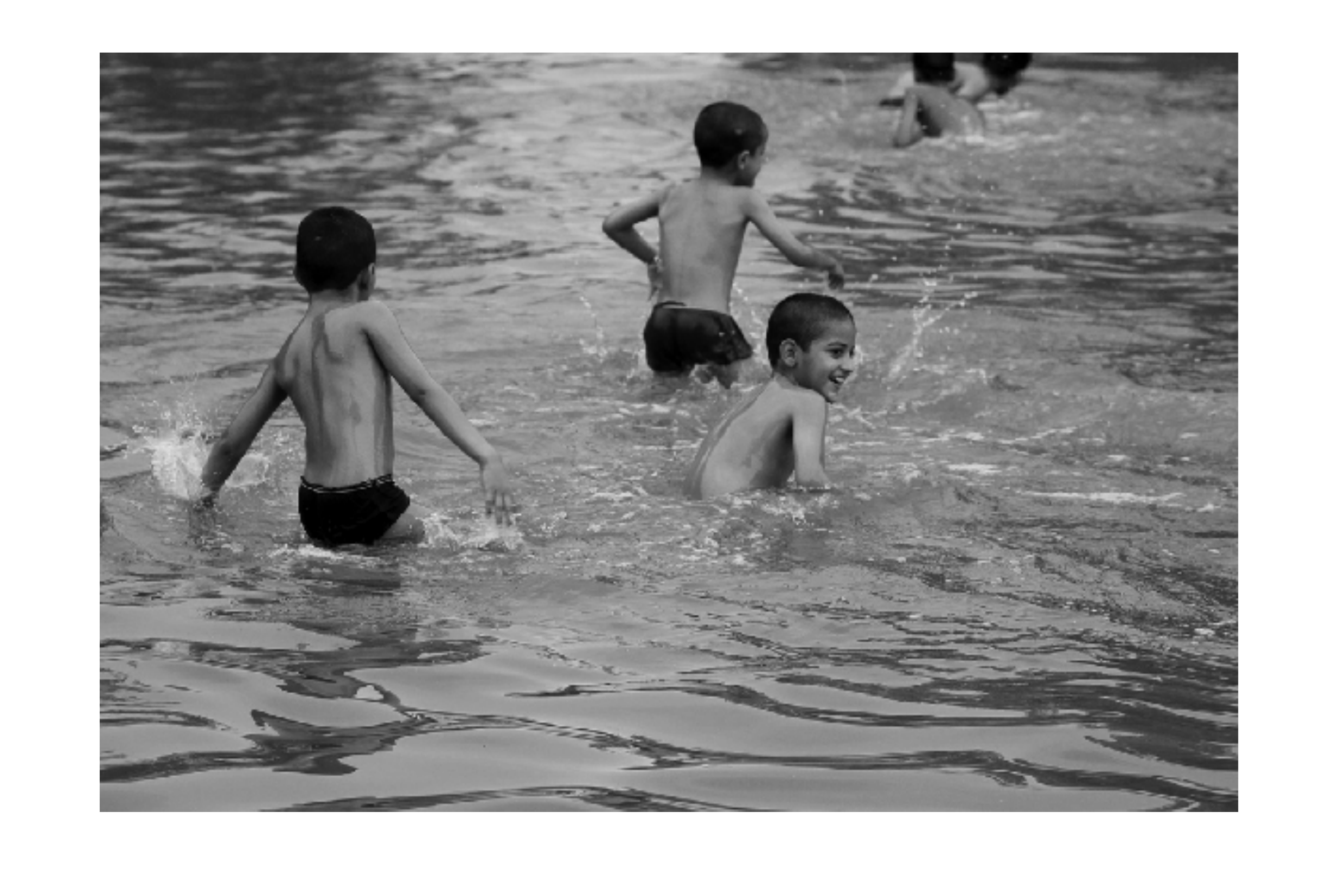} 
\includegraphics[width=0.30\textwidth, bb=0 0 700 700]{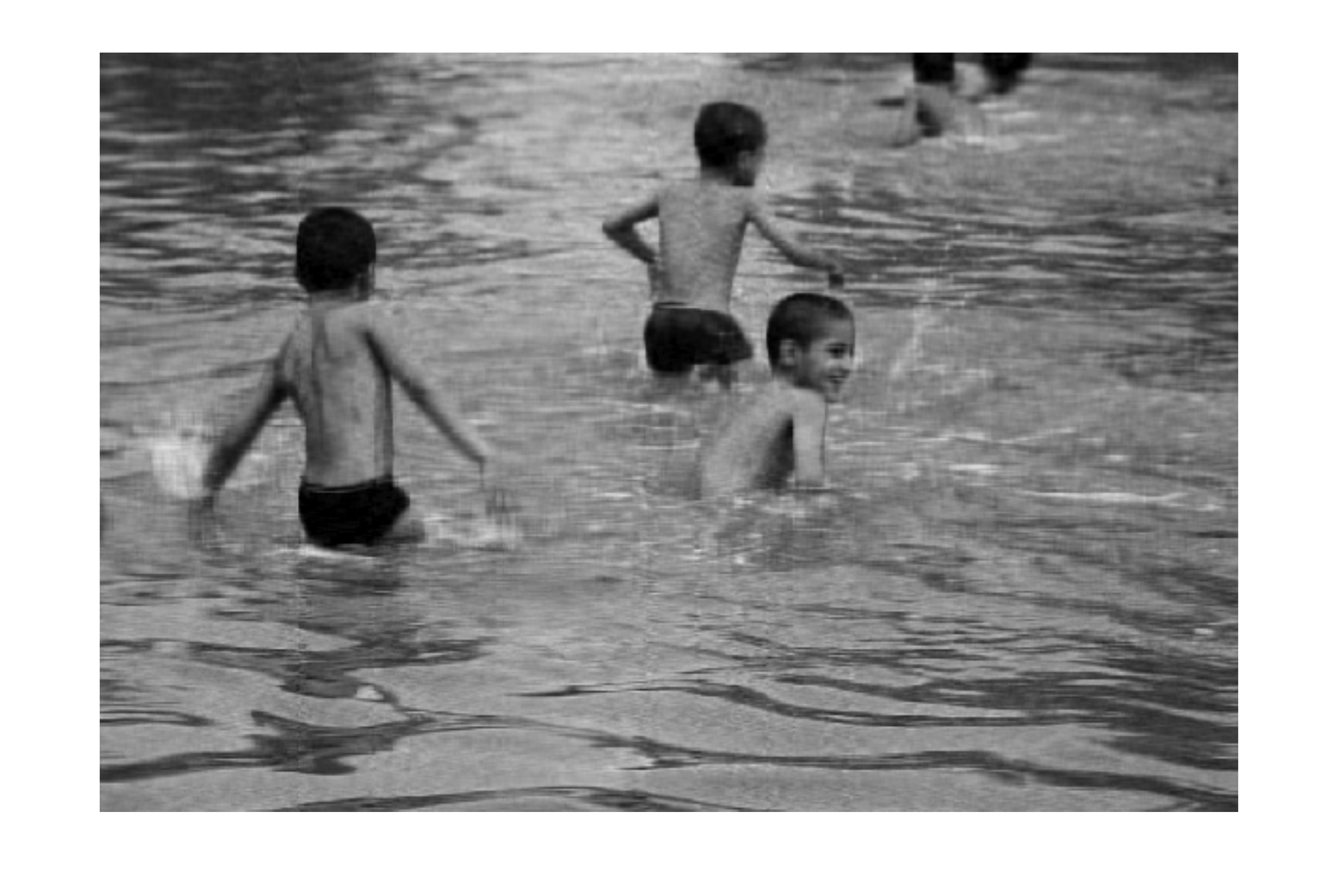} 
\includegraphics[width=0.30\textwidth, bb=0 0 700 700]{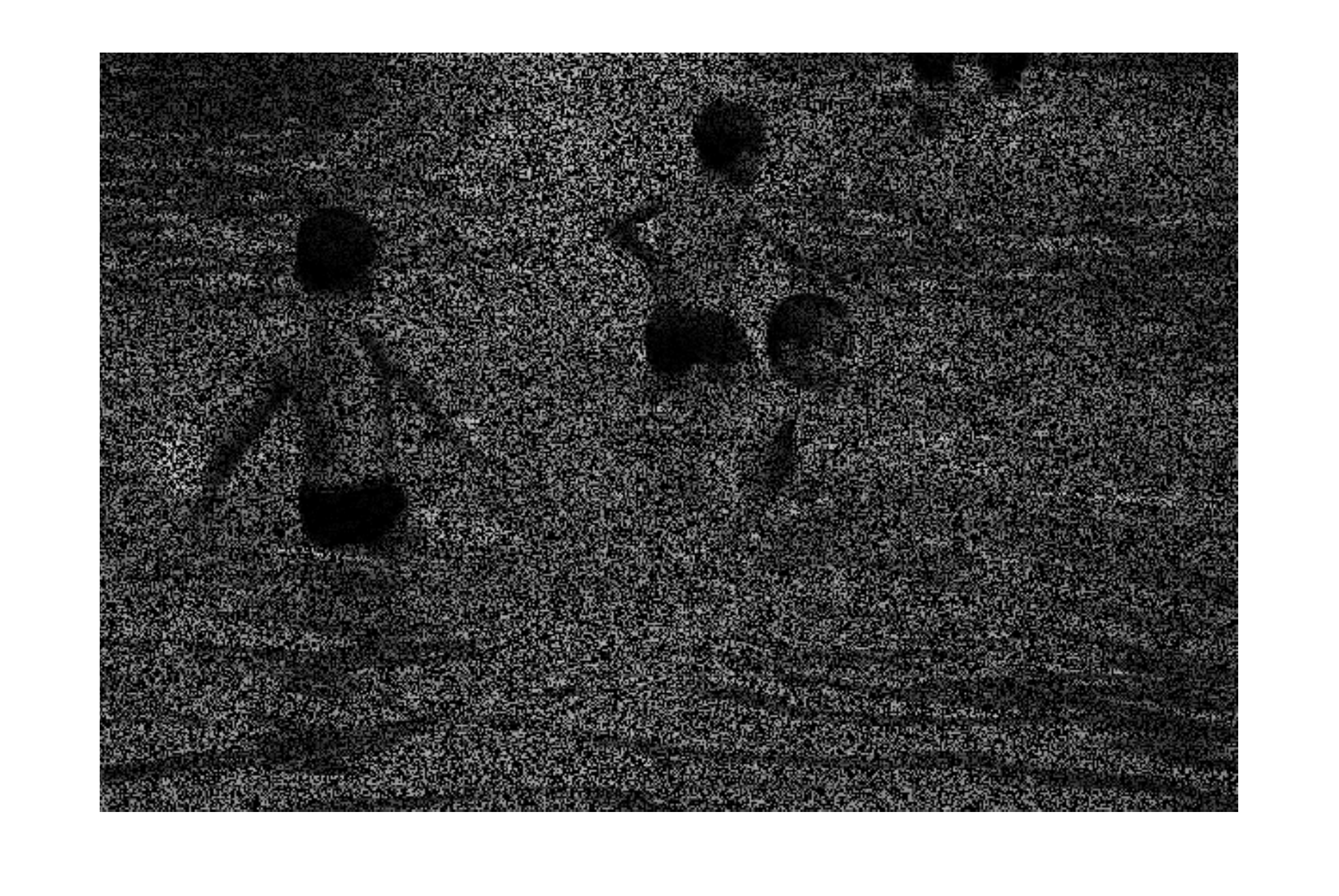} \\
\includegraphics[width=0.22\textwidth, bb=0 0 700 700]{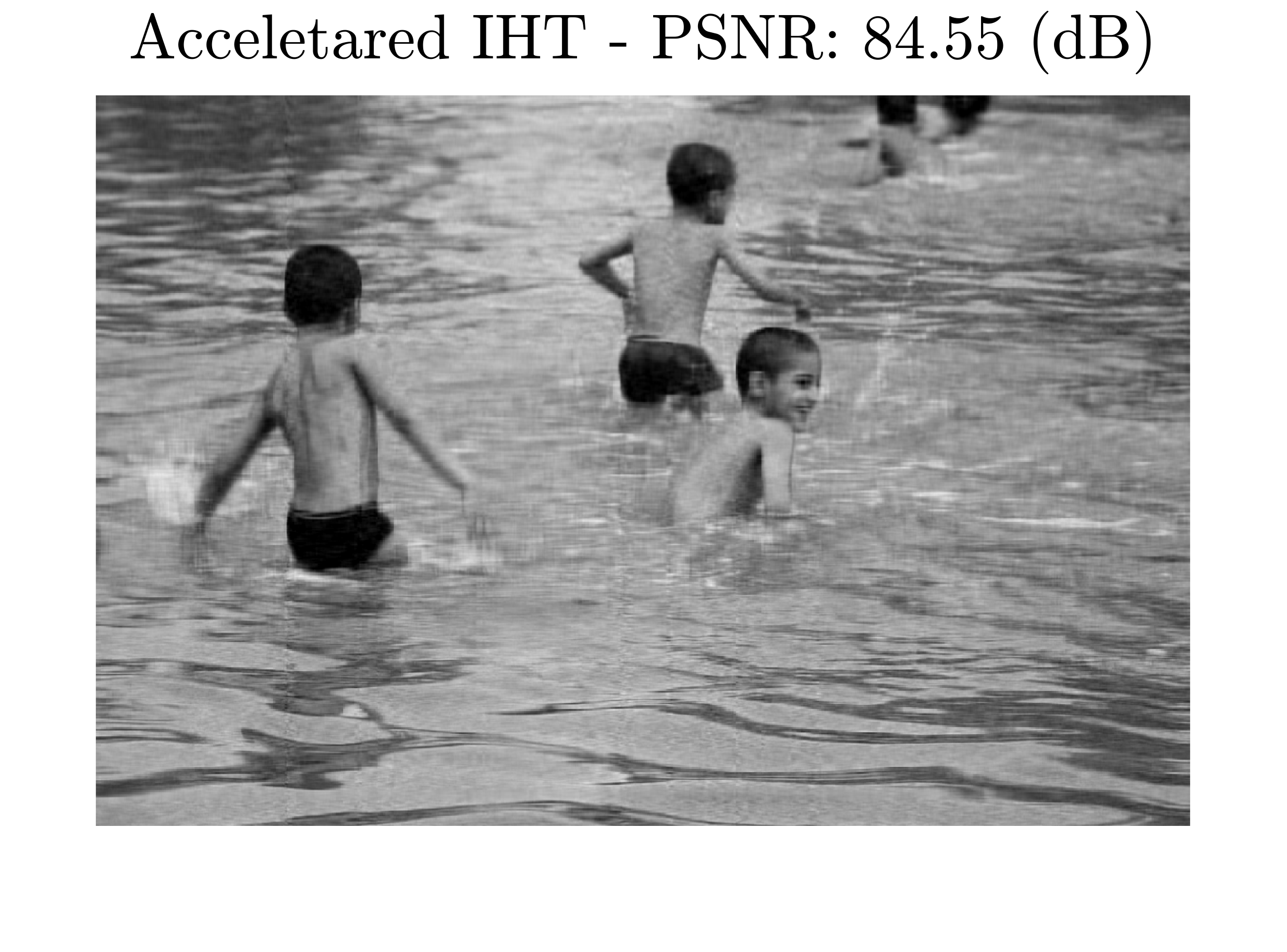} 
\includegraphics[width=0.22\textwidth, bb=0 0 700 700]{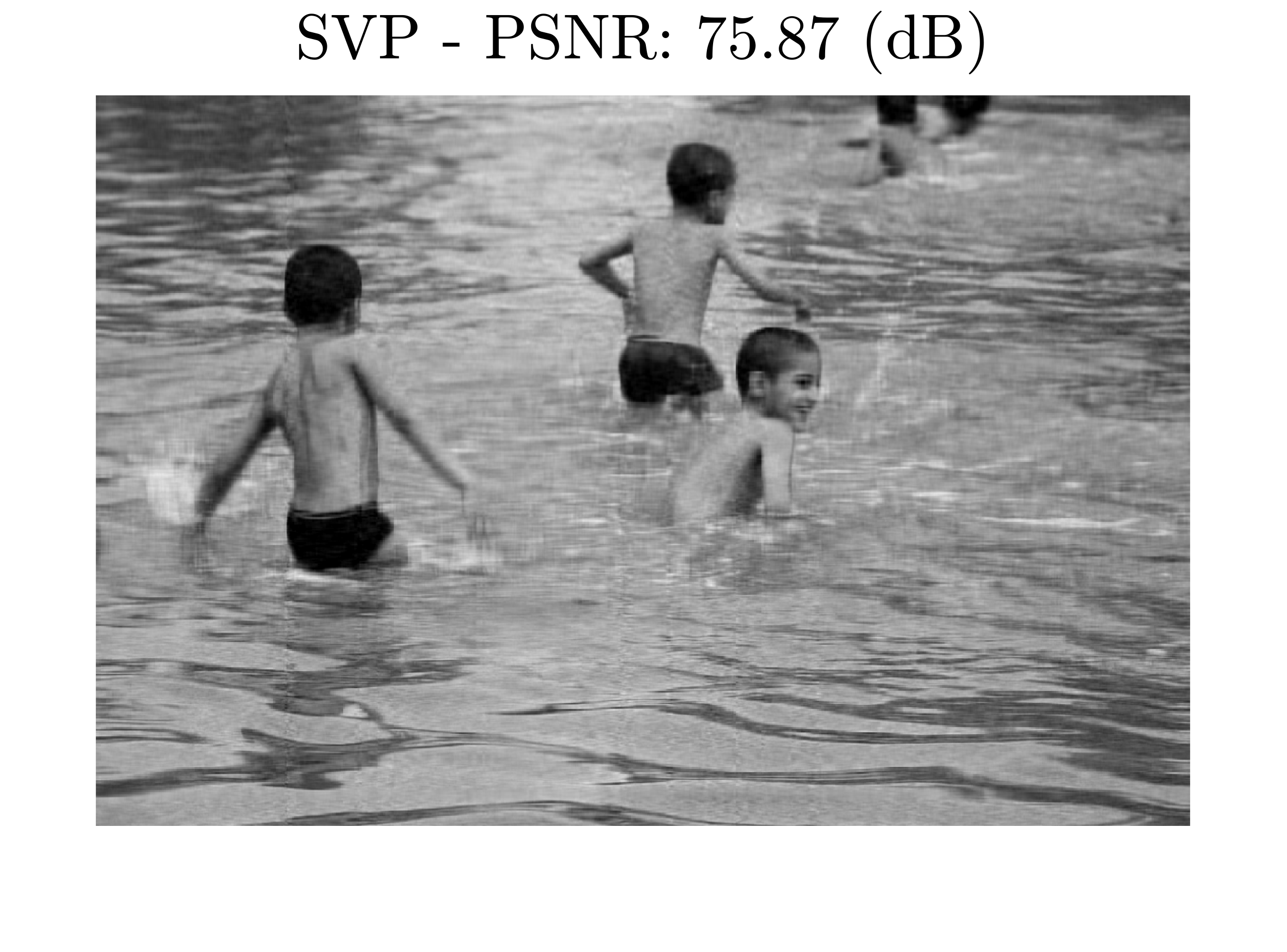} 
\includegraphics[width=0.22\textwidth, bb=0 0 700 700]{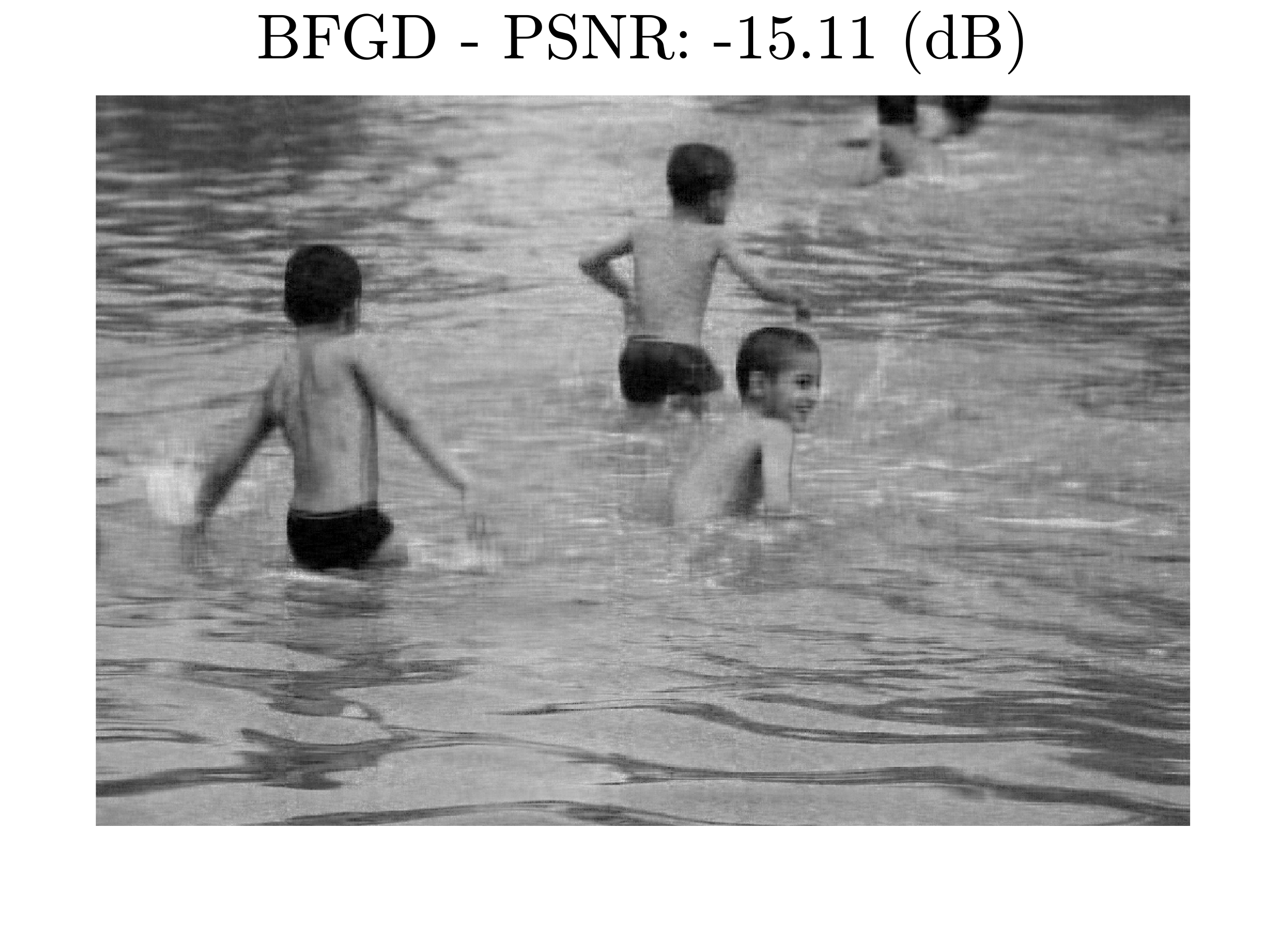} 
\includegraphics[width=0.22\textwidth, bb=0 0 700 700]{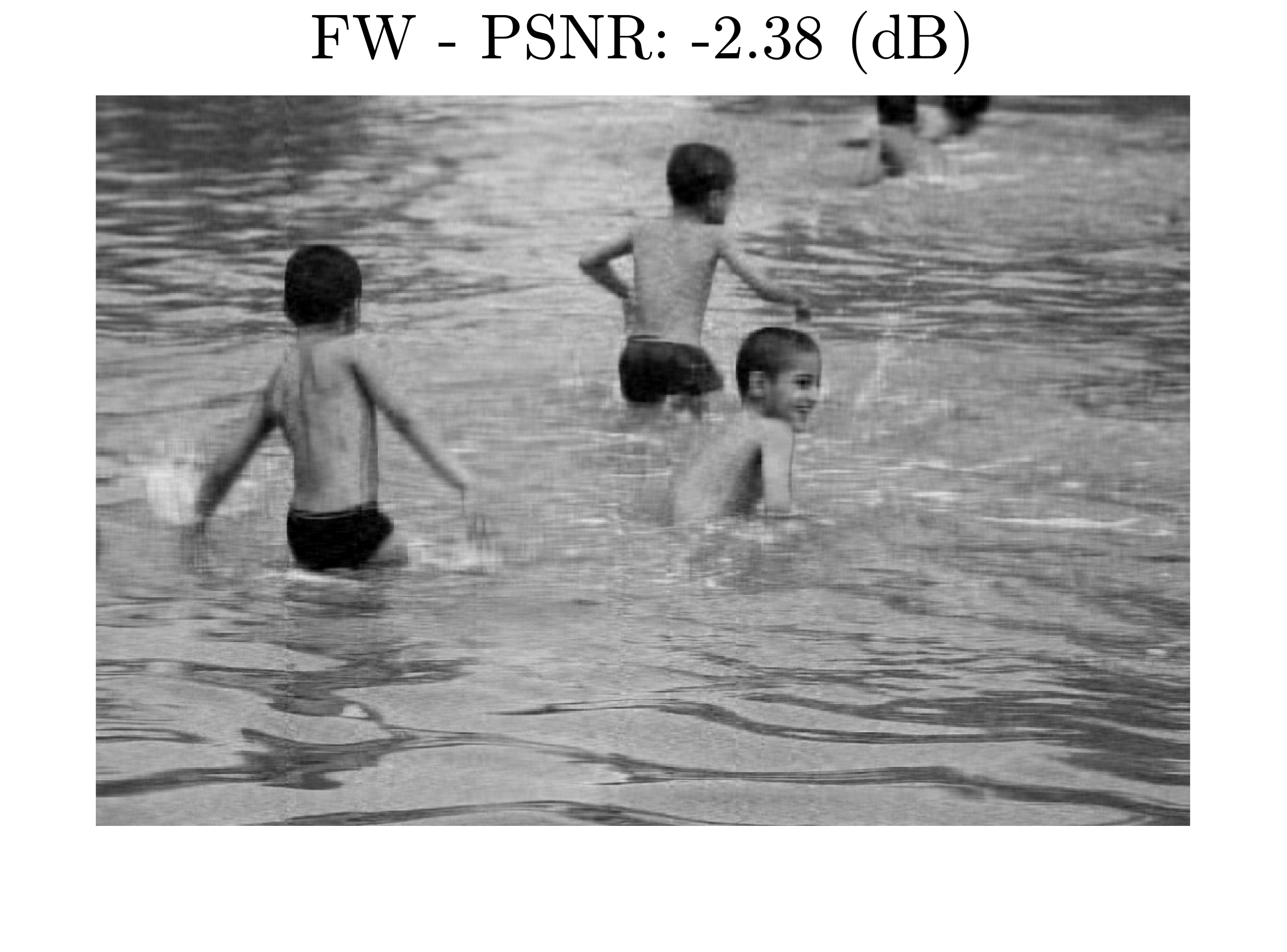}
\caption{Reconstruction performance in image denoising settings. The image size is $683 \times 1024$ ($699,392$ pixels) and the approximation rank is preset to $r = 60$. We observe 35\% of the pixels of the true image. \textbf{Top row:} Original, low rank approximation, and observed image. \textbf{Bottom row:} Reconstructed images. }
\label{fig:image2_MC}
\end{figure*}

%% file: AccIHT.bbl
\begin{thebibliography}{10}

\bibitem{nesterov1983method}
Y.~Nesterov.
\newblock A method of solving a convex programming problem with convergence
  rate {O}$(\tfrac{1}{k^2})$.
\newblock In {\em Soviet Mathematics Doklady}, volume~27, pages 372--376, 1983.

\bibitem{negahban2009unified}
S.~Negahban, B.~Yu, M.~Wainwright, and P.~Ravikumar.
\newblock A unified framework for high-dimensional analysis of $m$-estimators
  with decomposable regularizers.
\newblock In {\em Advances in Neural Information Processing Systems}, pages
  1348--1356, 2009.

\bibitem{chandrasekaran2012convex}
V.~Chandrasekaran, B.~Recht, P.~Parrilo, and A.~Willsky.
\newblock The convex geometry of linear inverse problems.
\newblock {\em Foundations of Computational mathematics}, 12(6):805--849, 2012.

\bibitem{bahmani2013greedy}
S.~Bahmani, B.~Raj, and P.~Boufounos.
\newblock Greedy sparsity-constrained optimization.
\newblock {\em Journal of Machine Learning Research}, 14(Mar):807--841, 2013.

\bibitem{clarkson2010coresets}
K.~Clarkson.
\newblock Coresets, sparse greedy approximation, and the {F}rank-{W}olfe
  algorithm.
\newblock {\em ACM Transactions on Algorithms (TALG)}, 6(4):63, 2010.

\bibitem{nesterov2013gradient}
Y.~Nesterov.
\newblock Gradient methods for minimizing composite functions.
\newblock {\em Mathematical Programming}, 140(1):125--161, 2013.

\bibitem{jaggi2013revisiting}
M.~Jaggi.
\newblock Revisiting {F}rank-{W}olfe: Projection-free sparse convex
  optimization.
\newblock In {\em ICML (1)}, pages 427--435, 2013.

\bibitem{lacoste2015global}
S.~Lacoste-Julien and M.~Jaggi.
\newblock On the global linear convergence of {F}rank-{W}olfe optimization
  variants.
\newblock In {\em Advances in Neural Information Processing Systems}, pages
  496--504, 2015.

\bibitem{blumensath2009iterative}
T.~Blumensath and M.~Davies.
\newblock Iterative hard thresholding for compressed sensing.
\newblock {\em Applied and computational harmonic analysis}, 27(3):265--274,
  2009.

\bibitem{jain2010guaranteed}
P.~Jain, R.~Meka, and I.~Dhillon.
\newblock Guaranteed rank minimization via singular value projection.
\newblock In {\em Advances in Neural Information Processing Systems}, pages
  937--945, 2010.

\bibitem{barber2017gradient}
R.~F. Barber and W.~Ha.
\newblock Gradient descent with nonconvex constraints: {L}ocal concavity
  determines convergence.
\newblock {\em arXiv preprint arXiv:1703.07755}, 2017.

\bibitem{kyrillidis2014matrix}
A.~Kyrillidis and V.~Cevher.
\newblock Matrix recipes for hard thresholding methods.
\newblock {\em Journal of mathematical imaging and vision}, 48(2):235--265,
  2014.

\bibitem{bach2012structured}
F.~Bach, R.~Jenatton, J.~Mairal, and G.~Obozinski.
\newblock Structured sparsity through convex optimization.
\newblock {\em Statistical Science}, pages 450--468, 2012.

\bibitem{jain2016structured}
P.~Jain, N.~Rao, and I.~Dhillon.
\newblock Structured sparse regression via greedy hard thresholding.
\newblock In {\em Advances in Neural Information Processing Systems}, pages
  1516--1524, 2016.

\bibitem{jain2014iterative}
P.~Jain, A.~Tewari, and P.~Kar.
\newblock On iterative hard thresholding methods for high-dimensional
  $m$-estimation.
\newblock In {\em Advances in Neural Information Processing Systems}, pages
  685--693, 2014.

\bibitem{kyrillidis2011recipes}
A.~Kyrillidis and V.~Cevher.
\newblock Recipes on hard thresholding methods.
\newblock In {\em Computational Advances in Multi-Sensor Adaptive Processing
  (CAMSAP), 2011 4th IEEE International Workshop on}, pages 353--356. IEEE,
  2011.

\bibitem{nesterov2013introductory}
Y.~Nesterov.
\newblock {\em Introductory lectures on convex optimization: {A} basic course},
  volume~87.
\newblock Springer Science \& Business Media, 2013.

\bibitem{polyak1964some}
B.~Polyak.
\newblock Some methods of speeding up the convergence of iteration methods.
\newblock {\em USSR Computational Mathematics and Mathematical Physics},
  4(5):1--17, 1964.

\bibitem{blumensath2012accelerated}
T.~Blumensath.
\newblock Accelerated iterative hard thresholding.
\newblock {\em Signal Processing}, 92(3):752--756, 2012.

\bibitem{agarwal2010fast}
A.~Agarwal, S.~Negahban, and M.~Wainwright.
\newblock Fast global convergence rates of gradient methods for
  high-dimensional statistical recovery.
\newblock In {\em Advances in Neural Information Processing Systems}, pages
  37--45, 2010.

\bibitem{beck2015minimization}
A.~Beck and N.~Hallak.
\newblock On the minimization over sparse symmetric sets: projections,
  optimality conditions, and algorithms.
\newblock {\em Mathematics of Operations Research}, 41(1):196--223, 2015.

\bibitem{williams1992the}
K.~Williams.
\newblock The $n$-th power of a $2 \times 2$ matrix.
\newblock {\em Mathematics Magazine}, 65(5):336, 1992.

\bibitem{kyrillidis2015structured}
A.~Kyrillidis, L.~Baldassarre, M.~El~Halabi, Q.~Tran-Dinh, and V.~Cevher.
\newblock Structured sparsity: Discrete and convex approaches.
\newblock In {\em Compressed Sensing and its Applications}, pages 341--387.
  Springer, 2015.

\bibitem{yuan2014gradient}
X.~Yuan, P.~Li, and T.~Zhang.
\newblock Gradient hard thresholding pursuit for sparsity-constrained
  optimization.
\newblock In {\em Proceedings of the 31st International Conference on Machine
  Learning (ICML-14)}, pages 127--135, 2014.

\bibitem{gong2013general}
P.~Gong, C.~Zhang, Z.~Lu, J.~Huang, and J.~Ye.
\newblock A general iterative shrinkage and thresholding algorithm for
  non-convex regularized optimization problems.
\newblock In {\em ICML (2)}, pages 37--45, 2013.

\bibitem{qiu2010ecme}
K.~Qiu and A.~Dogandzic.
\newblock {ECME} thresholding methods for sparse signal reconstruction.
\newblock {\em arXiv preprint arXiv:1004.4880}, 2010.

\bibitem{salakhutdinov2003adaptive}
R.~Salakhutdinov and S.~Roweis.
\newblock Adaptive overrelaxed bound optimization methods.
\newblock In {\em Proceedings of the 20th International Conference on Machine
  Learning (ICML-03)}, pages 664--671, 2003.

\bibitem{garg2009gradient}
R.~Garg and R.~Khandekar.
\newblock Gradient descent with sparsification: an iterative algorithm for
  sparse recovery with restricted isometry property.
\newblock In {\em Proceedings of the 26th Annual International Conference on
  Machine Learning}, pages 337--344. ACM, 2009.

\bibitem{blanchard2015cgiht}
J.~Blanchard, J.~Tanner, and K.~Wei.
\newblock {CGIHT}: {C}onjugate gradient iterative hard thresholding for
  compressed sensing and matrix completion.
\newblock {\em Information and Inference}, 4(4):289--327, 2015.

\bibitem{hestenes1952methods}
M.~Hestenes and E.~Stiefel.
\newblock {\em Methods of conjugate gradients for solving linear systems},
  volume~49.
\newblock NBS, 1952.

\bibitem{needell2009cosamp}
D.~Needell and J.~Tropp.
\newblock Co{S}a{MP}: Iterative signal recovery from incomplete and inaccurate
  samples.
\newblock {\em Applied and Computational Harmonic Analysis}, 26(3):301--321,
  2009.

\bibitem{foucart2011hard}
S.~Foucart.
\newblock Hard thresholding pursuit: an algorithm for compressive sensing.
\newblock {\em SIAM Journal on Numerical Analysis}, 49(6):2543--2563, 2011.

\bibitem{wei2015fast}
K.~Wei.
\newblock Fast iterative hard thresholding for compressed sensing.
\newblock {\em IEEE Signal Processing Letters}, 22(5):593--597, 2015.

\bibitem{bioucas2007new}
J.~Bioucas-Dias and M.~Figueiredo.
\newblock A new {TwIST}: Two-step iterative shrinkage/thresholding algorithms
  for image restoration.
\newblock {\em IEEE Transactions on Image processing}, 16(12):2992--3004, 2007.

\bibitem{beck2009fast}
A.~Beck and M.~Teboulle.
\newblock A fast iterative shrinkage-thresholding algorithm for linear inverse
  problems.
\newblock {\em SIAM journal on imaging sciences}, 2(1):183--202, 2009.

\bibitem{axelsson1996iterative}
O.~Axelsson.
\newblock {\em Iterative solution methods}.
\newblock Cambridge {U}niversity press, 1996.

\bibitem{schmidt2011convergence}
M.~Schmidt, N.~Roux, and F.~Bach.
\newblock Convergence rates of inexact proximal-gradient methods for convex
  optimization.
\newblock In {\em Advances in neural information processing systems}, pages
  1458--1466, 2011.

\bibitem{shalev2014accelerated}
S.~Shalev-Shwartz and T.~Zhang.
\newblock Accelerated proximal stochastic dual coordinate ascent for
  regularized loss minimization.
\newblock In {\em ICML}, pages 64--72, 2014.

\bibitem{becker2011templates}
S.~Becker, E.~Cand{\`e}s, and M.~Grant.
\newblock Templates for convex cone problems with applications to sparse signal
  recovery.
\newblock {\em Mathematical programming computation}, 3(3):165--218, 2011.

\bibitem{o2015adaptive}
B.~O’Donoghue and E.~Candes.
\newblock Adaptive restart for accelerated gradient schemes.
\newblock {\em Foundations of computational mathematics}, 15(3):715--732, 2015.

\bibitem{ghadimi2016accelerated}
S.~Ghadimi and G.~Lan.
\newblock Accelerated gradient methods for nonconvex nonlinear and stochastic
  programming.
\newblock {\em Mathematical Programming}, 156(1-2):59--99, 2016.

\bibitem{carmon2016accelerated}
Y.~Carmon, J.~Duchi, O.~Hinder, and A.~Sidford.
\newblock Accelerated methods for non-convex optimization.
\newblock {\em arXiv preprint arXiv:1611.00756}, 2016.

\bibitem{agarwal2016finding}
N.~Agarwal, Z.~Allen-Zhu, B.~Bullins, E.~Hazan, and T.~Ma.
\newblock Finding approximate local minima for nonconvex optimization in linear
  time.
\newblock {\em arXiv preprint arXiv:1611.01146}, 2016.

\bibitem{paquette2017catalyst}
C.~Paquette, H.~Lin, D.~Drusvyatskiy, J.~Mairal, and Z.~Harchaoui.
\newblock Catalyst acceleration for gradient-based non-convex optimization.
\newblock {\em arXiv preprint arXiv:1703.10993}, 2017.

\bibitem{li2015accelerated}
H.~Li and Z.~Lin.
\newblock Accelerated proximal gradient methods for nonconvex programming.
\newblock In {\em Advances in neural information processing systems}, pages
  379--387, 2015.

\bibitem{wilson2016lyapunov}
A.~Wilson, B.~Recht, and M.~Jordan.
\newblock A lyapunov analysis of momentum methods in optimization.
\newblock {\em arXiv preprint arXiv:1611.02635}, 2016.

\bibitem{scieur2017integration}
D.~Scieur, V.~Roulet, F.~Bach, and A.~d'Aspremont.
\newblock Integration methods and accelerated optimization algorithms.
\newblock {\em arXiv preprint arXiv:1702.06751}, 2017.

\bibitem{wolfe1970convergence}
P.~Wolfe.
\newblock Convergence theory in nonlinear programming.
\newblock {\em Integer and nonlinear programming}, pages 1--36, 1970.

\bibitem{mitchell1974finding}
B.~Mitchell, V.~Demyanov, and V.~Malozemov.
\newblock Finding the point of a polyhedron closest to the origin.
\newblock {\em SIAM Journal on Control}, 12(1):19--26, 1974.

\bibitem{garber2015faster}
D.~Garber and E.~Hazan.
\newblock Faster rates for the {F}rank-{W}olfe method over strongly-convex
  sets.
\newblock In {\em ICML}, pages 541--549, 2015.

\bibitem{donoho2005neighborliness}
D.~Donoho and J.~Tanner.
\newblock Neighborliness of randomly projected simplices in high dimensions.
\newblock {\em Proceedings of the National Academy of Sciences of the United
  States of America}, 102(27):9452--9457, 2005.

\bibitem{scikit-learn}
F.~Pedregosa, G.~Varoquaux, A.~Gramfort, V.~Michel, B.~Thirion, O.~Grisel,
  M.~Blondel, P.~Prettenhofer, R.~Weiss, V.~Dubourg, J.~Vanderplas, A.~Passos,
  D.~Cournapeau, M.~Brucher, M.~Perrot, and E.~Duchesnay.
\newblock Scikit-learn: Machine learning in {P}ython.
\newblock {\em Journal of Machine Learning Research}, 12:2825--2830, 2011.

\bibitem{das2011submodular}
A.~Das and D.~Kempe.
\newblock Submodular meets spectral: Greedy algorithms for subset selection,
  sparse approximation and dictionary selection.
\newblock In {\em Proceedings of the 28th International Conference on Machine
  Learning (ICML-11)}, pages 1057--1064, 2011.

\bibitem{zhang2009adaptive}
T.~Zhang.
\newblock Adaptive forward-backward greedy algorithm for sparse learning with
  linear models.
\newblock In {\em Advances in Neural Information Processing Systems}, pages
  1921--1928, 2009.

\bibitem{chen2001atomic}
S.~Chen, D.~Donoho, and M.~Saunders.
\newblock Atomic decomposition by basis pursuit.
\newblock {\em SIAM review}, 43(1):129--159, 2001.

\bibitem{recht2010guaranteed}
B.~Recht, M.~Fazel, and P.~Parrilo.
\newblock Guaranteed minimum-rank solutions of linear matrix equations via
  nuclear norm minimization.
\newblock {\em SIAM review}, 52(3):471--501, 2010.

\bibitem{tillmann2014computational}
A.~Tillmann and M.~Pfetsch.
\newblock The computational complexity of the restricted isometry property, the
  nullspace property, and related concepts in compressed sensing.
\newblock {\em IEEE Transactions on Information Theory}, 60(2):1248--1259,
  2014.

\bibitem{blumensath2011sampling}
T.~Blumensath.
\newblock Sampling and reconstructing signals from a union of linear subspaces.
\newblock {\em IEEE Transactions on Information Theory}, 57(7):4660--4671,
  2011.

\bibitem{shah2011iterative}
P.~Shah and V.~Chandrasekaran.
\newblock Iterative projections for signal identification on manifolds:
  {G}lobal recovery guarantees.
\newblock In {\em Communication, Control, and Computing (Allerton), 2011 49th
  Annual Allerton Conference on}, pages 760--767. IEEE, 2011.

\bibitem{kyrillidis2012combinatorial}
A.~Kyrillidis and V.~Cevher.
\newblock Combinatorial selection and least absolute shrinkage via the {CLASH}
  algorithm.
\newblock In {\em Information Theory Proceedings (ISIT), 2012 IEEE
  International Symposium on}, pages 2216--2220. IEEE, 2012.

\bibitem{hegde2015approximation}
C.~Hegde, P.~Indyk, and L.~Schmidt.
\newblock Approximation algorithms for model-based compressive sensing.
\newblock {\em IEEE Transactions on Information Theory}, 61(9):5129--5147,
  2015.

\bibitem{jacob2009group}
L.~Jacob, G.~Obozinski, and J.-P. Vert.
\newblock Group lasso with overlap and graph lasso.
\newblock In {\em Proceedings of the 26th annual international conference on
  machine learning}, pages 433--440. ACM, 2009.

\bibitem{negahban2012restricted}
S.~Negahban and M.~Wainwright.
\newblock Restricted strong convexity and weighted matrix completion: {O}ptimal
  bounds with noise.
\newblock {\em Journal of Machine Learning Research}, 13(May):1665--1697, 2012.

\bibitem{bhojanapalli2016dropping}
S.~Bhojanapalli, A.~Kyrillidis, and S.~Sanghavi.
\newblock Dropping convexity for faster semi-definite optimization.
\newblock In {\em Conference on Learning Theory}, pages 530--582, 2016.

\bibitem{park2016finding}
D.~Park, A.~Kyrillidis, C.~Caramanis, and S.~Sanghavi.
\newblock Finding low-rank solutions to matrix problems, efficiently and
  provably.
\newblock {\em arXiv preprint arXiv:1606.03168}, 2016.

\end{thebibliography}
